\documentclass{article}

\usepackage[T1]{fontenc}
\usepackage{amsmath,amssymb,amsthm}
\usepackage{breqn}
\usepackage{color}

\newtheorem{thm}{Theorem}
\newtheorem{prop}{Proposition}
\newtheorem{lemme}{Lemma}
\newtheorem{cor}{Corollary}

\newtheorem{déf}{Definition}
\newtheorem{rem}{Remark}
\newtheorem{notation}{Notation}

\title{Toric generalized Kähler structures}

\author{Laurence Boulanger}

\begin{document}
\maketitle

\setlength{\parskip}{0.05cm}

\begin{abstract}
Given a compact symplectic toric manifold $(M,\omega, \mathbb{T})$, we identify a class $DGK_{\omega}^{\mathbb{T}}(M)$ of $\mathbb{T}$-invariant generalized Kähler structures for which a generalisation the Abreu-Guillemin theory of toric Kähler metrics holds. Specifically, elements of $DGK_{\omega}^{\mathbb{T}}(M)$ are characterized by the data of a strictly convex function $\tau$ on the moment polytope associated to $(M,\omega, \mathbb{T})$ via the Delzant theorem, and an antisymmetric matrix $C$. For a given $C$, it is shown that a toric Kähler structure on $M$ can be explicitly deformed to a non-Kähler element of $DGK_{\omega}^{\mathbb{T}}(M)$ by adding a small multiple of $C$. This constitutes an explicit realization of a recent unobstructedness theorem of R. Goto \cite{goto:1, goto:2}, where the choice of a matrix $C$ corresponds to choosing a holomorphic Poisson structure. Adapting methods from S. K. Donaldson \cite{donaldson:2}, we compute the moment map for the action of $\mathrm{Ham}(M,\omega)$ on $DGK_{\omega}^{\mathbb{T}}(M)$. The result introduces a natural notion of "generalized Hermitian scalar curvature". In dimension 4, we find an expression for this generalized Hermitian scalar curvature in terms of the underlying bi-Hermitian structure in the sense of Apostolov-Gauduchon-Grantcharov \cite{apostolov:4}.
\end{abstract}

\section{Introduction}

This paper is concerned with the theory of generalized Kähler structures as defined and studied by M.~Gualtieri in \cite{gualtieri:1} in the context of N.~Hitchin's \cite{hitchin:1} generalized complex geometry. Our goal is to identify a natural notion of scalar curvature for a generalized Kähler structure. The approach we use to study this problem draws 
from the following three ingredients. 

(1) The first concerns the interpretation of the scalar curvature as a moment map. Given a compact symplectic $2m$-manifold $(M,\omega)$, the space $AK_{\omega}(M)$ of $\omega$-compatible almost complex structures on $M$ is a Fréchet manifold endowed with a natural formal Kähler structure. A.~Fujiki and S.~K.~Donaldson oberved that the group $\mathrm{Ham}(M,\omega)$ of hamiltonian diffeomorphisms acts on $AK_{\omega}(M)$ in a hamiltonian fashion, and that the moment map can be identified with the Hermitian scalar curvature $u_J$ of the almost Hermitian structure $(\omega,J)$ as follows. Recall that $u_J$ is defined as
\begin{equation}\label{uAlternative}
u_J=\frac{2m\rho\wedge \omega^{m-1}}{\omega^m},
\end{equation}
where $\rho$ is the real curvature 2-form of the hermitian connection induced on the anticanonical bundle of $(M,J)$ by the Chern connection of $(\omega,J)$.

\begin{thm}[\cite{fujiki:1,donaldson:1}]\label{Theorem 1}
Let $C^{\infty}_0(M)$ be the space of smooth functions on $M$ with zero mean, identified to the Lie algebra $\mathfrak{ham}(M,\omega)$ via the Poisson bracket. Then the expression
\begin{equation}\label{ThmDonaldson}
\nu^f(J):=-\int_Mfu_J\frac{\omega^m}{m!}
\end{equation}
is the moment map for the natural action of $\mathrm{Ham}(M,\omega)$ on $AK_{\omega}(M)$.
\end{thm}

The reader can consult \cite{gauduchon:2} for a detailed proof.

(2) The second ingredient is the computation by S.~K.~Donaldson of this moment map in the context of the Abreu-Guillemin theory of toric Kähler metrics. Let $(M,\omega,\mathbb{T})$ be a symplectic toric $2m$-manifold with moment map $\mu:M\rightarrow \mathfrak{t}^*$ and let $K_{\omega}^{\mathbb{T}}(M)$ be the subspace of $\mathbb{T}$-invariant $\omega$-compatible complex structures. In his seminal work on toric Kähler structures, V.~Guillemin discovered that the elements $J\in K_{\omega}^{\mathbb{T}}(M)$ can be described, up to $\mathbb{T}$-invariant biolomorphisms, in terms of convex functions on the interior of the moment polytope $\Delta$ for $(M,\omega, \mathbb{T})$ as follows.

\begin{thm}[\cite{guillemin:1}]\label{ThmGuillemin}
 For any $J\in K_{\omega}^{\mathbb{T}}(M)$ and any given choice of basis $(\xi_1,\ldots,\xi_m)$ of $\mathfrak{t}$, there exists momentum-angle coordinates $(\mu^j,t^j)$ on $\mathring{M}$ such that
$$
\omega =  \sum\limits_{j=1}^md\mu^j\wedge dt^j
$$
and
$J$ is of the anti-diagonal form 
\begin{equation}\label{FormeAntidiagoKähler}
J\frac{\partial}{\partial \mu^j} = \sum\limits_{k=1}^m\Psi_{jk}\frac{\partial}{\partial t^k}, \quad \Psi_{jk} = \frac{\partial^2\tau}{\partial\mu^j\partial\mu^k},
\end{equation}
where $\tau=\tau(\mu^1,\ldots,\mu^m)$ is a strictly convex smooth function defined on $\mathring{\Delta}$. Conversely, for any smooth strictly convex function $\tau$ on $\mathring{\Delta}$, formula \eqref{FormeAntidiagoKähler} defines an element of $K_{\omega}^{\mathbb{T}}(\mathring{M})$.
\end{thm}

For this reason, the function $\tau$ is often referred to as the {\em symplectic potential} of $J$ in the literature \cite{donaldson:2}. In \cite{abreu:2}, M.~Abreu discovered that the scalar curvature $u_J$ of the Riemannian metric associated to $J\in K_{\omega}^{\mathbb{T}}(M)$ is given by the formula
\begin{equation}\label{AbreuFormula}
u_J=-\sum\limits_{i,j=1}^m \frac{\partial^2\tau^{ij}}{\partial\mu^i\partial\mu^j}.
\end{equation}
Here, $(\tau^{ij}) = (\mathrm{Hess}(\tau)^{-1})_{ij}$. Equation \eqref{AbreuFormula} is commonly known as {\em Abreu's formula}. S.~K.~Donaldson \cite{donaldson:2} observed that Theorem \ref{Theorem 1} combined with the description \eqref{FormeAntidiagoKähler} of elements in $K_{\omega}^{\mathbb{T}}(M)$ gives an alternative way for deriving \eqref{AbreuFormula}, by directly showing that \eqref{AbreuFormula} computes the moment map for the action of $\mathrm{Ham}^{\mathbb{T}}(M,\omega)$ on $K_{\omega}^{\mathbb{T}}(M)$. This observation suggested a similar form for the Hermitian scalar curvature of elements in $AK_{\omega}^{\mathbb{T}}(M)$ which has been checked directly by M.~Lejmi \cite{lejmi:1}.

(3) The third ingredient is the notion of generalized Kähler structure of symplectic type and their realization as $\omega$-tamed complex structures. Recall that a generalized almost complex structure on a smooth $2m$-manifold $M$ is a complex structures $\mathcal{J}$ on the vector bundle $TM \oplus T^*M$ which is orthogonal with respect to the natural inner product $\langle X\oplus\xi, Y\oplus\eta\rangle = \frac{1}{2}(\xi(Y)+\eta(X))$. A generalized complex structure is a generalized almost complex structure satisfying the integrability condition
$$
[\mathcal{J}U,\mathcal{J}V]_C-\mathcal{J}[\mathcal{J}U,V]_C-\mathcal{J}[U,\mathcal{J}V]_C-[U,V]_C=0
$$
with respect to the Courant bracket
$$
[X\oplus\xi,Y\oplus\eta]_C = [X,Y] + \mathcal{L}_X\eta -\mathcal{L}_Y\xi -\frac{1}{2}d(\iota_X\eta - \iota_Y\xi).
$$
Denote by $GAC(M)$ and $GC(M)$ the sets of generalized almost complex and generalized complex structures on $M$ respectively. For example \cite{gualtieri:1}, if $\omega$ is a symplectic form on $M$, then $\mathcal{J}_{\omega}:X\oplus\xi \mapsto -\omega^{-1}(\xi)\oplus \omega(X)$ defines an element of $GC(M)$. Following \cite{gualtieri:1}, a generalized almost Kähler structure on $M$ is defined as a pair $(\mathcal{J}_1,\mathcal{J}_2)$ of elements of $GAC(M)$ such that
\begin{itemize}
 \item[(1)] $\mathcal{J}_1\mathcal{J}_2 = \mathcal{J}_2\mathcal{J}_1$
 \item[(2)] $\langle -\mathcal{J}_1\mathcal{J}_2\cdot,\cdot\rangle >0.$
\end{itemize}
On a symplectic manifold $(M,\omega)$, we thus introduce the spaces $GAK_{\omega}(M)$, $GK_{\omega}(M)$ of generalized almost Kähler (resp. generalized Kähler) structures of symplectic type. These are defined by
$$
GAK_{\omega}(M)=\{\mathcal{J}\in GAC(M) \ | \ \mathcal{J}_{\omega}\mathcal{J}=\mathcal{J}\mathcal{J}_{\omega}, \langle-\mathcal{J}_{\omega}\mathcal{J}\cdot,\cdot\rangle>0\},
$$
$$
GK_{\omega}(M)=GAK_{\omega}(M)\cap GC(M).
$$
As a trivial example, if $(J,\omega)$ is a genuine Kähler structure on $M$, then $\mathcal{J}_J\in GK_{\omega}(M)$ where $\mathcal{J}_J$ is the generalized complex structure associated to $J$ by $\mathcal{J}_J:X\oplus\xi\mapsto JX\oplus J\xi$.

One can endow the space $GAK_{\omega}(M)$ with a formal Kähler structure such that $\mathrm{Ham}(M,\omega)$ acts symplectically on it. Thus, a moment map for this action, if it exists, could be interpreted as a scalar curvature by virtue of Theorem \ref{Theorem 1}. In order to compute this moment map, we specialize to the case of a compact symplectic toric manifold $(M,\omega,\mathbb{T})$ with moment map $\mu:M\rightarrow\Delta\subset\mathfrak{t}^*$ and Delzant polytope $\Delta$. Let $GAK^{\mathbb{T}}_\omega(M)$ denote the $\mathbb{T}$-invariant elements of $GAK_\omega(M)$.

Following Donaldson's argument in \cite{donaldson:2}, we compute the moment map and obtain a generalization of Abreu's formula as follows. 

\begin{thm}[cf. Theorem \ref{ThmMomapActionRéduite}]\label{ThmMomapActionRéduite}
Denote by $C^{\infty}_{c,0}(M)^{\mathbb{T}}$ the set of $\mathbb{T}$-invariant functions with zero mean supported in $\mathring{M}$ and by $\mathrm{Ham}_c^{\mathbb{T}}(M,\omega)$ the subgroup of hamiltonian diffeomorphisms that it generates. The action of $\mathrm{Ham}_c^{\mathbb{T}}(M,\omega)$ on $GAK_{\omega}^{\mathbb{T}}(M)$ is hamiltonian with moment map $\nu:GAK^{\mathbb{T}}_\omega(M)\rightarrow (C^{\infty}_{c,0}(M)^{\mathbb{T}})^*$ given by
\begin{equation}\label{ExprAlt}
\nu^f(\mathcal{J})=-\int_{\mathring{M}}f\left(\sum\limits_{i,j=1}^m \frac{\partial^2\tilde{Q}_{ij}}{\partial \mu^i\partial\mu^j}\right)\frac{\omega^m}{m!},
\end{equation}
where $\tilde{Q}_{ij}= \omega\left(\frac{\partial}{\partial t^i},A\frac{\partial}{\partial t^j}\right)$ and $A$ is the $\mathrm{End}(TM)$-part of $\mathcal{J}$.
\end{thm}

In light of this result, we are led to define the {\em generalized Hermitian scalar curvature} of $\mathcal{J}\in GAK^{\mathbb{T}}_\omega(M)$ to be the function
$$
u_{GK}(\mathcal{J})=\sum\limits_{i,j=1}^m \frac{\partial^2\tilde{Q}_{ij}}{\partial \mu^i\partial\mu^j}.
$$

We further investigate this formula when $\mathcal{J}$ is restricted to a certain class $DGK_{\omega}^{\mathbb{T}}(M)$ of generalized Kähler metrics such that $K_{\omega}^{\mathbb{T}}(M) \subsetneq DGK_{\omega}^{\mathbb{T}}(M)\subsetneq GK_{\omega}^{\mathbb{T}}(M)$ (cf. Section \ref{SectionPotentiel} for the precise definition). For this class, we prove the following generalization of Theorem \ref{ThmGuillemin}:

\begin{thm}[cf. Theorem \ref{PropPrincipale}]\label{ThmGuilleminGénéralisé}
For any $\mathcal{J}\in DGK_{\omega}^{\mathbb{T}}(M)$ and any choice of basis $(\xi_1,\ldots,\xi_m)$ of $\mathfrak{t}$, there exist momentum-angle coordinates $(\mu^j,t^j)$ on $\mathring{M}$ such that 
$$
\omega =  \sum\limits_{j=1}^md\mu^j\wedge dt^j
$$
and $\mathcal{J}$ is determined by an $(m\times m)$-matrix-valued smooth function $\Psi$ of the form 
$$
\Psi_{jk} = \frac{\partial^2\tau}{\partial\mu^j\partial\mu^k}+C_{jk}
$$
for a smooth strictly convex function $\tau$ on $\mathring{\Delta}$ and a (constant) antisymmetric matrix $C$ (cf. Section \ref{SectionPotentiel} for details). Conversely, to any antisymmetric matrix $C$ and smooth strictly convex function $\tau$ on $\mathring{\Delta}$, there corresponds an element $\mathring{\mathcal{J}}\in DGK_{\omega}^{\mathbb{T}}(\mathring{M})$.
\end{thm}

Besides this, the class $DGK_{\omega}(M)$ is interesting in its own right in the context of generalized Kähler geometry because of the following compactification theorem:

\begin{thm}[cf. Theorem \ref{ThmC123}]\label{ThmCompactIntro}
Consider $\mathcal{J}\in DGK^{\mathbb{T}}_{\omega}(M)$ corresponding to a matrix $\Psi$ in the sense of Theorem \ref{ThmGuilleminGénéralisé} and $\mathring{\mathcal{J}}\in DGK^{\mathbb{T}}_\omega(\mathring{M})$ corresponding to a matrix $\mathring{\Psi}$ with respect to the $(\mu^j,t^j)$ coordinates associated with $\mathcal{J}$. 
If 
\begin{itemize}
 \item[(C1)] $\mathring{\Psi}-\Psi$ admits a smooth extension to $\Delta$;
 \item[(C2)] $\Psi^T\mathring{\Psi}^{-1}\Psi-\Psi^T$ admits a smooth extension to $\Delta$;
 \item[(C3)] $\beta+\sum\limits_{i,j=1}^m(\mathring{\Psi}-\Psi)_{ij}d\mu^i\otimes d\mu^j+(\Psi^T\mathring{\Psi}^{-1}\Psi-\Psi^T)_{ij}(Jd\mu^i)\otimes(J d\mu^j)$ is positive definite on $M\backslash \mathring{M}$;
\end{itemize}
then $\mathring{\mathcal{J}}$ is the restriction of an element of $DGK^{\mathbb{T}}_\omega(M)$.
\end{thm}

\begin{cor}[cf. Corollary \ref{CorDéformation}]\label{CorollaireDeDéformation}
Let $J_0\in K_{\omega}^{\mathbb{T}}(M)$ be an $\omega$-compatible complex structure of the form \eqref{FormeAntidiagoKähler}. Given an antisymmetric matrix $C$, define a family of matrix-valued functions $\mathring{\Psi}(t)$ ($t\in \mathbb{R}$) on $\mathring{\Delta}$ by 
$$
\mathring{\Psi}_{jk}(t)=\frac{\partial^2\tau}{\partial\mu^j\partial\mu^k}+tC_{jk},
$$
and let $\mathring{\mathcal{J}}_t\in DGK^{\mathbb{T}}_\omega(\mathring{M})$ be the corresponding family of generalized complex structures (in the sense of Theorem \ref{ThmGuilleminGénéralisé}). For sufficiently small values of $|t|$, the family $\mathring{\mathcal{J}}_t$ is the restriction to $\mathring{M}$ of a family $\mathcal{J}_t\in DGK_\omega^{\mathbb{T}}(M)$. 
\end{cor} 

This manner of deforming a Kähler structure into a generalized Kähler can be viewed as an explicit realization of a recent unobstructedness theorem of R.~Goto \cite{goto:1, goto:2}, where the matrix $C$ corresponds to choosing a holomorphic Poisson structure $\sigma$ in the setting of \cite{goto:1} (See Proposition \ref{propGoto}).

Using our newly found notion of generalized Hermitian scalar curvature, we generalize E.~Calabi's notion of extremal Kähler metrics, calling {\em extremal} any element $\mathcal{J}$ of $GAK_{\omega}^{\mathbb{T}}(M)$ which is a critical point of the functional $\mathcal{J}\mapsto \int_Mu_{GK}(\mathcal{J})^2\frac{\omega^m}{m!}$. We deduce, as it is done in \cite{abreu:1} in the Kähler setting, that $\mathcal{J}\in DGK_{\omega}^{\mathbb{T}}(M)$ is extremal if and only if $u_{GK}(\mathcal{J})$ is an affine function of the momenta. This, and Corollary \ref{CorollaireDeDéformation}, provide examples of extremal strictly generalized Kähler metrics obtained as deformations of extremal Kähler toric varieties. See \cite{donaldson:2,donaldson:3,chen:1,zhou:1} for a general theory.

In the case of a compact symplectic toric manifold of dimension 4, we are able to prove in Theorem \ref{PropCompactificationDim4} that the compactification conditions (C1), (C2) of Theorem \ref{ThmCompactIntro} are actually necessary. In the 4-dimensional context, we also derive a closed form expression for the generalized Hermitian scalar curvature of elements in $DGK_{\omega}^{\mathbb{T}}(M)$ in terms of the classical scalar curvature (cf. Corollary \ref{CoroS_GK}). This result confirms the form of the generalized scalar curvature suggested in \cite{coimbra:1} and gives an exact value to the dilaton $\phi$ in terms of the angle between the complex structures of the underlying Hermitian structures.

\vspace{5mm}

{\em Acknowledgements.} This present paper is based on material originally from my PhD thesis. I wish to thank my supervisor Vestislav Apostolov for sharing his time and ideas so generously. I also thank Paul Gauchon for accepting to share some of his personal notes with me and Marco Gualtieri whose suggestions have helped to better the presentation of this paper.

\section{Generalized Kähler structures of symplectic type}\label{section3}
In this section, we introduce the notion of generalized almost Kähler structure of symplectic type which is the main object of the paper. We provide three characterizations of these structures which will be used throughout this paper depending on the situation. We shall also define a formal symplectic structure on the space of generalized almost Kähler structures with respect to which the action of the group of hamiltonian diffeomorphisms is symplectic.

Recall that \cite{hitchin:3} a {\em generalized complex structure} on a smooth manifold $M$ is a complex structure $\mathcal{J}$ on the vector bundle $TM\oplus T^*M$ which is orthogonal with respect to the natural inner product $\langle X\oplus\xi, Y\oplus\eta\rangle = \frac{1}{2}(\xi(Y)+\eta(X))$ and which satisfies the integrability condition 
$$
[\mathcal{J}U,\mathcal{J}V]_C-\mathcal{J}[\mathcal{J}U,V]_C-\mathcal{J}[U,\mathcal{J}V]_C-[U,V]_C=0
$$
with respect to the {\em Courant bracket} 
$$
[X\oplus\xi,Y\oplus\eta]_C = [X,Y] + \mathcal{L}_X\eta -\mathcal{L}_Y\xi -\frac{1}{2}d(\iota_X\eta - \iota_Y\xi).
$$
If the integrability condition is omited, we refer to $\mathcal{J}$ as a generalized {\em almost} complex structure. For instance, if $\omega$ is a non-degenerate 2-form and $J$ is an almost complex structure, then the endomorphisms of $TM\oplus T^*M$ 
\begin{flalign*}
& \mathcal{J}_{\omega}:X\oplus\xi \mapsto -\omega^{-1}(\xi)\oplus \omega(X),\\
& \mathcal{J}_{J}:X\oplus\xi \mapsto JX\oplus J\xi,
\end{flalign*}
are generalized almost complex structures. The integrability of $\mathcal{J}_J$ is equivalent to the usual integrability of $J$, while the integrability of $\mathcal{J}_{\omega}$ is equivalent to $d\omega=0$. A pair $(\mathcal{J}_1,\mathcal{J}_2)$ of generalized almost complex structures such that $\mathcal{J}_1\circ \mathcal{J}_{2} = \mathcal{J}_{2}\circ\mathcal{J}_1$ and the bilinear form $\langle-\mathcal{J}_1\mathcal{J}_2\cdot,\cdot\rangle$ is positive definite is called a {\em generalized almost Kähler structure}. It is a {\em generalized Kähler structure} provided both $\mathcal{J}_{1}$ and $\mathcal{J}_{2}$ are integrable. As a trivial example, if $(\omega, J)$ is a genuine Kähler structure on $M$, then $(\mathcal{J}_{\omega},\mathcal{J}_J)$ is generalized Kähler.

\begin{rem}
{\rm
The structure group of a generalized almost Kähler structure is $U(m)\times U(m)\subset U(m,m)$ which is maximal compact (cf. \cite{gualtieri:1}).
}
\end{rem}

It turns out \cite{gualtieri:1} that a generalized almost Kähler structure $(\mathcal{J}_1,\mathcal{J}_2)$ on $M$ is equivalent to the data $(J_+,J_-,g,b)$ of a Riemannian metric $g$, a 2-form $b$ and two $g$-compatible almost complex structures $J_+,J_-$. 

Indeed, the involution $-\mathcal{J}_1\mathcal{J}_2$ induces a splitting of $TM\oplus T^*M$ into its ($\pm1$)-eigenbundles $C_{\pm}$. The bilinear form $\langle\cdot,\cdot\rangle$ is then positive definite on $C_+$ and negative definite on $C_-$. On the one hand this implies that $C_{\pm}$ are both of dimension $2m$, and on the other that $C_{\pm}\cap TM = C_\pm\cap T^*M=0$ (since $TM$ and $T^*M$ are isotropic in $TM\oplus T^*M$). It follows that $C_+$ is the graph of a map $TM\rightarrow T^*M$ whose symetric and antisymetric parts we denote by $g$ and $b$ respectively. Similarly $C_-$ is the graph of $b-g$ and we have isomorphisms $\iota_{\pm}:TM\rightarrow C_\pm:X\mapsto (X, \iota_X(b\pm g))$. The generalized almost complex structrures $\mathcal{J}_1,\mathcal{J}_2$ preserve $C_{\pm}$ and so we may use $\iota_{\pm}$ to transfer them to almost complex structures $J_{\pm}$ on $TM$:
\begin{equation}\label{eqGualtieri2}
\begin{aligned}
J_+:=& \iota_+^{-1}\circ \mathcal{J}_1\circ \iota_+ = \iota_+^{-1}\circ \mathcal{J}_2\circ \iota_+, \\
J_- :=& \iota_-^{-1}\circ \mathcal{J}_1\circ \iota_- = -\iota_-^{-1}\circ \mathcal{J}_2\circ \iota_-.
\end{aligned}
\end{equation}
In fact, if $\iota_+$ is used to transfer $\langle\cdot,\cdot\rangle|_{C_+}$ on $TM$, we obtain precisely $g$. It follows that the pairs $(J_{\pm},g)$ are almost Hermitian structures. Explicitely, the generalized almost Kähler structure $(\mathcal{J}_1,\mathcal{J}_2)$ is given in terms of $(J_+,J_-,g,b)$ by
Explicitely, 
\begin{equation}\label{eqGualtieri}
\begin{aligned}
\mathcal{J}_1=\frac{1}{2}
e^b
\left( \begin{array}{cc}
J_+ + J_- & -(F_+^{-1}-F_-^{-1}) \\
F_+ - F_- & J_+^*+J_-^*
\end{array} \right)
e^{-b}, \\
\mathcal{J}_2=\frac{1}{2}
e^b
\left( \begin{array}{cc}
J_+ - J_- & -(F_+^{-1}+F_-^{-1}) \\
F_+ + F_- & J_+^*-J_-^*
\end{array} \right)
e^{-b}.
\end{aligned}
\end{equation}
Here, $F_{\pm}=g(J_{\pm}\cdot,\cdot)$ are the fundamental 2-forms of the Hermitian structures $(J_{\pm},g)$ and $e^b$ is the automorphisms of $TM\oplus T^*M$ given by $X\oplus\xi\mapsto X\oplus b(X)+\xi$. The integrability of $(\mathcal{J}_1,\mathcal{J}_2)$ is then equivalent to the integrability of $J_+$ and $J_-$ together with the relation
\begin{equation}\label{IntégrabilitéGualtieri}
d^c_{\pm}F_{\pm}=\mp db,
\end{equation}
where $d^c_{\pm}$ is the operator $J_{\pm}dJ^{-1}_{\pm}$ for the action of $J_{\pm}$ on $p$-forms by $J_{\pm}\psi=(-1)^p\psi(J_{\pm}\cdot,\ldots,J_{\pm}\cdot)$.

\begin{déf}
Given a symplectic form $\omega$ on $M$, denote by $GAK_{\omega}(M)$ the space of almost complex structures $\mathcal{J}$ such that $(\mathcal{J}_{\omega},\mathcal{J})$ is a generalized almost Kähler structure. We shall refer to the elements of $GAK_{\omega}(M)$ as {\bf generalized almost Kähler structures of symplectic type}. The set of {\em integrable} elements of $GAK_{\omega}(M)$ will be denoted by $GK_{\omega}(M)$.
\end{déf}

\begin{rem}
{\rm
If $\mathcal{J}\in GK_{\omega}(M)$, then $(\mathcal{J}_{\omega}, \mathcal{J})$ is a generalized Kähler structure since $\omega$ is symplectic.
}
\end{rem}

Recall that an almost complex structure $J$ is called {\em $\omega$-tamed} if the bilinear form $\omega(\cdot,J\cdot)$ is positive definite. Let us denote $AC_+(M,\omega)$ the set of all $\omega$-tamed almost complex structures on $M$. The following is well known (see for instance \cite{enrietti:1}):

\begin{prop}\label{PropEnrietti}
The correspondence $GAK_{\omega}(M)\rightarrow AC_+(M,\omega):\mathcal{J}\mapsto J_+$ given by \eqref{eqGualtieri2} is bijective. The inverse map is $J\mapsto (J_+,J_-,g,b)$, where
$$
J_+=J,\quad J_-=J^{*_{\omega}},\quad g=-\frac{1}{2}\omega (J - J^{*_{\omega}}),\quad b=-\frac{1}{2}\omega (J + J^{*_{\omega}}),
$$
for $J^{*_{\omega}}=-\omega^{-1}J^*\omega$ the {\em symplectic adjoint} of $J$. Moreover, $\mathcal{J}\in GAK_{\omega}(M)$ is integrable if and only if $J_+$ and $J_+^{*_\omega}$ are integrable.
\end{prop}

Note that in this context, the Kähler case corresponds to taking $J$ integrable and $\omega$-compatible (in which case $J_-=-J_+$).

The material in the remainder of this section is adapted from unpublished notes of P. Gauduchon \cite{gauduchon:1}. Let $(M,\omega)$ be a {\em compact} symplectic manifold of real dimension $2m$. We denote by $v_\omega=\omega^m\slash m!$ the symplectic volume form. It is straightforward to check that $\mathcal{J}\in GAK_{\omega}(M)$ if and only if $\mathcal{J}$ is of the form
$$
\mathcal{J}=
\left(
\begin{array}{cc}
A & B\omega^{-1} \\
-\omega B & A^*
\end{array}
\right),
$$
where $A,B$ are endomorphisms of $TM$ satisfying
\begin{equation} \label{conditionsAB}
\begin{aligned}
 A^2-B^2=-\mathrm{Id},\\
 AB+BA=0, \\
 A^{*_\omega}=-A, \\
 B^{*_\omega}=B,
\end{aligned}
\end{equation}
as well as the positivity relation
\begin{equation}\label{positivitéAB}
\omega(X,AX)+\omega(Y,AY)+2\omega(BX,Y)>0 \quad \forall X,Y\in TM.
\end{equation}
In terms of the corresponding $J\in AC_+(M,\omega)$, we have
\begin{equation}\label{J-AB}
A = -2(J-J^{*_\omega})^{-1}, \quad B= -(J+J^{*_\omega})(J-J^{*_\omega})^{-1}.
\end{equation}
Equations \eqref{conditionsAB} suggests a complex description of the situation. Indeed, if we define an endomorphism $K=A+iB$ of $T^{\mathbb{C}}M=TM\otimes \mathbb{C}$, then the first two equations are equivalent to $K^2=-\mathrm{Id}$, while the other two are equivalent to $K^{*_\omega}=-\overline{K}$. To express the positivity condition, it is natural to introduce the (non-degenerate anti-Hermitian) bilinear form $H(U,V)=\omega(U,\overline{V})$. Indeed, one may easily verify that \eqref{positivitéAB} is then equivalent to positivity of $H_K=H(\cdot,K\cdot)$. Viewing $GAK_{\omega}(M)$ as the set of such complex endomorphisms, we endow it with the structure of a Fréchet manifold with a formal symplectic structure in a manner analogous to \cite{gauduchon:2}. Indeed, the tangent space at $K$ is given by
\begin{equation}\label{GAK-tangent}
T_K(GAK_\omega(M))= \{\dot{K}\in C^{\infty}(\mathrm{End}(T^{\mathbb{C}}M)) \ | \ \dot{K}^{*_{\omega}}=-\dot{\overline{K}}, \ \dot{K}K+K\dot{K}=0\},
\end{equation}
and the symplectic form is
$$
\Omega_K(\dot{K}_1,\dot{K}_2)=\frac{1}{2}\int_M\mathrm{tr}(K\dot{K}_1\dot{K}_2)v_{\omega}.
$$
\begin{rem}
{\rm
\begin{itemize}
 \item [(1)] It is straightfoward to check that for any $K\in GAK_\omega(M)$ and $\dot{K} \in T_K(GAK_\omega(M))$, we have $K\dot{K}\in T_K(GAK_\omega(M))$. Using this and the fact that the elements of $T_K(GAK_\omega(M))$ are symmetric with respect to the Hermitian scalar product $H_K$, we see that $\Omega$ is indeed real and positive definite. In fact, if we define a formal almost complex structure by $\mathbb{K}_K\dot{K}:=K\dot{K}$, it can be shown that the pair $(\Omega,\mathbb{K})$ defines a formal Kähler structure on $GAK_{\omega}(M)$.
\item[(2)] Note that $GAK_\omega(M)$ naturally contains the set $AK_\omega(M)$ of almost Kähler structures as a symplectic submanifold by considering the {\em real} elements of $GAK_\omega(M)$ (\i.e. $\mathfrak{Im}K=0$). In fact, the restriction of $\Omega$ to $AK_\omega(M)$ is the symplectic form considered by A. Fujiki \cite{fujiki:1}. 
\end{itemize}
}
\end{rem}

Before going further, recall that a {\em hamiltonian vector field} $X_f$ on $(M,\omega)$ is a symplectic vector field of the form $X_f=-\omega^{-1}df$ for a function $f\in C^{\infty}(M)$. It is also called the {\em symplectic gradient} of $f$, and will be sometimes denoted by $\mathrm{grad}_{\omega}f$. The group $\mathrm{Ham}(M,\omega)$ of {\em hamiltonian diffeomorphisms} is the set of 1-parameter subgroups generated by vector fields of the form $\mathrm{grad}_{\omega}f_t$ for $f\in C^{\infty}(M\times I,\mathbb{R})$ where $I$ is an interval containing 0. According to a result of A. Banyaga \cite{banyaga:1}, to every path $t\mapsto\gamma_t\in\mathrm{Ham}(M,\omega)$ through the identity corresponds a family of functions $f_t\in C^{\infty}(M)$ such that
$$
\frac{d\gamma_t}{dt}=\mathrm{grad}_{\omega}f_t \quad \forall t.
$$
It follows that the Lie algebra of $\mathrm{Ham}(M,\omega)$ can be identified with the hamiltonian vector fields on $(M,\omega)$:
$$
\mathfrak{ham}(M,\omega)=\{\mathrm{grad}_{\omega}f\in C^{\infty}(TM)\ | \ f\in C^{\infty}(M)\}.
$$
The exponential map is given by the flow:
$$
\exp(tX)=\varphi^X_t; \quad \frac{d}{dt}\varphi^X_t= X\circ \varphi^X_t, \quad \varphi^X_0=\mathrm{Id}_M.
$$
In turn, this Lie algebra is identified to the space $C_0^{\infty}(M)$ of smooth functions $f$ normalized by the condition $\int_M fv_\omega=0$, and endowed with the Poisson bracket $\{f,g\}=X_f\cdot g = -X_g\cdot f$. The correspondence being 
\begin{equation}\label{ham-Fonctions}
C^{\infty}_0(M)\rightarrow \mathfrak{ham}(M,\omega):f\mapsto \mathrm{grad}_{\omega}f.
\end{equation}
It is also possible to use the $\mathrm{Ad}$-invariant euclidean scalar product $(f,g)=\int_M fgv_{\omega}$ to identify $C^{\infty}_0(M)$ to a subset of $C^{\infty}_0(M)^*$.

The group $\mathrm{Ham}(M,\omega)$ acts on $GAK_\omega(M)$ by $\varphi\cdot K=\varphi_*K\varphi_*^{-1}$ and the infinitesimal action corresponding to $V\in\mathfrak{ham}(M,\omega)$ is given by
\begin{equation}\label{Vfondamental}
V^{\sharp}_{\; K} = \left.\frac{d}{dt}\right|_{t=0}\varphi^V_t\cdot K=-\mathcal{L}_VK.
\end{equation} 
As shown in Theorem \ref{Theorem 1}, the restriction of this action to $AK_\omega(M)$, is hamiltonian, and the moment map can be identified with the Hermitian scalar curvature $u_J$.

\section{Toric Generalized Kähler structures}

In this section, we study generalized Kähler structures of symplectic type on compact symplectic toric manifolds. Section \ref{SectionDelzant} recalls the elements of symplectic toric geometry, which will be used in this paper. A source for this material is the monograph \cite{guillemin:2}. Section \ref{SectionPotentiel} introduces the class $DGK_{\omega}^{\mathbb{T}}(M)$ of torus-invariant anti-diagonal generalized Kähler structures of symplectic type, and we show that elements in this class are parametrized by the data of an antisymmetric matrix $C$ and a strictly convex smooth function $\tau$ defined on the interior of the moment polytope. This generalizes the notion of symplectic potential discovered by V. Guillemin \cite{guillemin:1} and M. Abreu \cite{abreu:1,abreu:2} in the Kähler setting. In section \ref{Compactification et déformation}, we adress the question of compactification, which is to determine whether a given pair $(\tau,C)$ as above comes from an element of $DGK_{\omega}^{\mathbb{T}}(M)$. In the spirit of \cite{apostolov:3}, we list sufficient conditions for compactification, and as a corollary, we obtain a simple and explicit procedure for deforming a toric Kähler metric to a strictly generalized Kähler element of $DGK_{\omega}^{\mathbb{T}}(M)$.

\subsection{Delzant theory}\label{SectionDelzant}
Recall that a compact {\em symplectic toric manifold} of dimension $2m$ is a triple $(M,\omega,\mathbb{T})$ such that the torus $\mathbb{T}$ of dimension $m$ acts on the compact connected symplectic manifold $(M,\omega)$ of real dimension $2m$ in an effective and hamiltonian fashion with moment map $\mu:M\rightarrow \mathfrak{t}^*:x\mapsto(\mu(x):\xi\mapsto \mu^{\xi}(x))$. In turn, this means that $\mu$ is $\mathbb{T}$-equivariant (in fact {\em $\mathbb{T}$-invariant} as $\mathbb{T}$ is abelian) and for all $\xi\in\mathfrak{t}=\mathrm{Lie}(\mathbb{T})$, $\mu^{\xi}$ is a hamiltonian function for the infinitesimal action $\xi^{\sharp}$ induced  on $M$ by $\xi$. According to M. F. Atiyah \cite{atiyah:1} and Guillemin-Sternberg \cite{guillemin:3}, the image $\Delta=\mu(M)\subset\mathfrak{t}^*$ of the moment map is the convex hull of the image by $\mu$ of the fixed points of the action. A theorem of T. Delzant \cite{delzant:1} states that compact symplectic toric manifolds are classified (up to equivariant symplectomorphisms) by their moment polytopes $\Delta$. Recall the definition of these classifying polytopes:
\begin{déf}\label{DéfDelzantPolytope}
Let $\mathfrak{t}$ be a vector space of dimension $m$. A {\em Delzant polytope} with $d$ facets in $\mathfrak{t}^*$ is the data $(\Delta,\Lambda,\nu_1,\ldots,\nu_d)$ of a set $\Delta\subset\mathfrak{t}^*$ which is the convex hull of a finite number of points called {\em vertices}, a lattice $\Lambda\subset\mathfrak{t}$ and {\em normals} $\nu_1,\ldots,\nu_d\in\Lambda$ such that
$$
\Delta=\{x\in \mathfrak{t}^* \ | \ L_j(x)\geq 0, \ j=1,\ldots, d\},
$$
where the $L_j$'s are functions of the form
$$
L_j(x)=\langle \nu_j, x\rangle+\lambda_j 
$$
for certain numbers $\lambda_1,\ldots,\lambda_d\in \mathbb{R}$, and such that for each vertex $x\in\Delta$, the normals $\nu_j$ for which $L_j(x)=0$ make up a basis of $\Lambda$. The {\em facets} of $\Delta$ are the sets $F_j$ of the form
$$
F_j= \{x\in \Delta \ | \ L_j(x)=0 \}, \quad j=1,\ldots,d.
$$
A {\em face} of codimension $k$ of $\Delta$ is the intersection of $k$ facets. For a face $F$, we call {\em interior} of $F$ the set $\mathring{F}$ of points of $F$ which are in no face of smaller codimension. In other words, if $F=\bigcap_{j\in I}^kF_{j}$ for a certain set in indices $I=\{j_1,\ldots,j_k\}$, then
$$
\mathring{F} = \{x\in \Delta \ | \ L_j(x)=0 \Leftrightarrow j\in I \}.
$$
\end{déf}

It is shown in Delzant \cite{delzant:1} that for any face $F=F_{j_1}\cap\ldots\cap F_{j_k}$ of codimension $k$ and any $p\in \mu^{-1}(\mathring{F})$, the stabilizer of $p$ in $\mathbb{T}$ is the sub-torus $\mathbb{T}_F$ of dimension $k$ corresponding to the subalgebra $\mathfrak{t}_F$ generated by the normals $\nu_{j_1},\ldots,\nu_{j_k}$. Moreover, $M_F=\mu^{-1}(F)$ is a symplectic toric submanifold of codimension $2k$ for the action of $\mathbb{T}\slash \mathbb{T}_F$. Its moment polytope is naturally identified with $F$, in the following sense. The face $F$ is supported by an affine subspace of the form $x_0+\mathfrak{t}_F^0$, where $\mathfrak{t}_F^0\cong (\mathfrak{t}\slash\mathfrak{t}_F)^*$ is the annihilator of $\mathfrak{t}_F$ in $\mathfrak{t}^*$. A moment map for the effective action of $\mathbb{T}\slash \mathbb{T}_F$ is then $\mu|_{M_F}-x_0$. The preimage 
$$
\mathring{M}:=\mu^{-1}(\mathring{\Delta})
$$
of the interior of the moment polytope corresponds to the set of points where the action of $\mathbb{T}$ is free, and this set is open and dense in $M$ (cf. \cite{ginzburg:1} Corollary B.48). Finally, let us mention the observation in \cite{lerman:1} (Proposition 7.3) that the set of smooth functions $C^{\infty}(\Delta)$ (\i.e. those functions which are the restriction to $\Delta$ of a function of $C^{\infty}(\mathfrak{t}^*)$) is pulled back to $M$ via $\mu$ to the set $C^{\infty}(M)^{\mathbb{T}}$ of smooth $\mathbb{T}$-invariant functions. Because of this, we shall freely identify $C^{\infty}(M)^{\mathbb{T}}$ and $C^{\infty}(\Delta)$.

\subsection{The symplectic potential}\label{SectionPotentiel}

Let $(M,\omega, \mathbb{T})$ be a compact symplectic toric manifold of real dimension $2m$, with moment map $\mu:M\rightarrow \Delta\subset \mathfrak{t}^*$. In this section, we are concerned with the generalized almost Kähler structures of symplectic type on $(M,\omega)$ (cf. section \ref{section3}) which are invariant under the action of $\mathbb{T}$. In accordance with the identification in Proposition \ref{PropEnrietti}, such a structure can also be regarded as an $\omega$-tamed $\mathbb{T}$-invariant almost complex structure on $M$. Recall also that such a $J$ represents an integrable generalized almost Kähler structure if and only if both $J$ and $J^{*_\omega}$ are integrable. 

\begin{notation}
{\rm
Let $GAK_{\omega}^{\mathbb{T}}(M)$ (resp. $GK_{\omega}^{\mathbb{T}}(M)$) denote the set of $\mathbb{T}$-invariant generalized almost Kähler (resp. generalized Kähler) structures of symplectic type as defined in section \ref{section3}. Similarly, let $AK_{\omega}^{\mathbb{T}}(M)$ (resp. $K_{\omega}^{\mathbb{T}}(M)$) denote the set of $\mathbb{T}$-invariant $\omega$-compatible almost complex (resp. complex) structures.
}
\end{notation}

\begin{prop}\label{LemmeCoordonnéesCX} Let $J\in GK_{\omega}^{\mathbb{T}}(M)$. Given a basis $(\xi_1,\ldots, \xi_m)$ of $\mathfrak{t}$ and $K_i=\xi_i^{\sharp}$ the corresponding infinitesimal actions on $M$, there exists pluriharmonic functions $u^j$ on $\mathring{M}$ which, in a neighborhood of each point, can be completed by angular coordinates $t^j$ to form a system $J$-holomorphic coordinates $(u^j,t^j)$ such that
$$
\frac{\partial}{\partial u^j} = -JK_j, \ \frac{\partial}{\partial t^j} = K_j.
$$
Moreover, for each such coordinate system, we may replace the functions \linebreak $u^1,\ldots,u^m$ by the functions $\mu^1,\ldots,\mu^m$ (where $\mu^i=\mu^{\xi_i}$) to obtain new coordinates $(\mu^j,t^j)$.
\end{prop}

\begin{proof}
Everywhere on $M$, the tangent space to the orbits of the action is generated by the infinitesimal actions $K_1,\ldots, K_m$. In particular, on $\mathring{M}$ where the orbits are of dimension $m$, the $K_i$'s are linearly independent. Denote by $\mathcal{K}$ the Lagrangian distribution on $\mathring{M}$ generated by the $K_i$. Then, we have $\mathcal{K}\oplus J\mathcal{K}=T\mathring{M}$. To complete the proof, it suffices to show that the Lie bracket of each pair of basis elements $(K_1,\ldots,K_m,JK_1,\ldots,JK_m)$ vanishes. The multiplicative structure on $\mathfrak{t}$ being trivial, we have $[K_i,K_j]=[\xi_i,\xi_j]^{\sharp}=0$ $\forall i,j$. Next, since the action of $\mathbb{T}$ preserves $J$, we have $\mathcal{L}_{K_i}J=0$ $\forall i$, which is equivalent to $[K_i,J\cdot]=J[K_i,\cdot]$. In particular, $[K_i,JK_j]=J[K_i,K_j]=0$. Finally, $J$ being integrable, the missing equality $[JK_i,JK_j]=0$ follows from the vanishing of the Nijenhuis tensor $N_J(X,Y)=[JX,JY]-J[JX,Y]-J[X,JY]-[X,Y]$. We deduce from this that $\mathrm{span}(d\mu^1,\ldots,d\mu^m) = \mathrm{span}(du^1,\ldots,du^m)$, and so the functions $(\mu^1,\ldots,\mu^m,t^1,\ldots,t^m)$ define a local diffeomorphism.
\end{proof}

It is important to note that even though the functions $u^j$ and $t^j$ are only defined locally, the coordinate fields $\frac{\partial}{\partial u^j}$, $\frac{\partial}{\partial t^j}$, $\frac{\partial}{\partial \mu^j}$ (as well as the 1-forms $du^j$, $dt^j$, $d\mu^j$) are well-defined globally on $\mathring{M}$ for a fixed choice of a basis $(\xi_j)$ of $\mathfrak{t}$. {\em From now on, we fix once and for all a basis $(\xi_j)$ of $\mathfrak{t}$ and we denote $(x^j)$  the coordinates on $\mathfrak{t}^*$ induced by the dual basis $(\xi_j^*)$}. 

For $J\in K_{\omega}^{\mathbb{T}}(M)$, it is well known that the coordinates $(\mu^j,t^j)$ from Proposition \ref{LemmeCoordonnéesCX} are symplectic \cite{apostolov:3}. They are refered to as momentum-angle coordinates associated to $J$. However, for a general $J\in GK_{\omega}^{\mathbb{T}}(M)$, we shall see in Propostion \ref{LemmeCoordonnéesS} below that this is only the case if the symplectic dual $J^{*_{\omega}}$ is "anti-diagonal" in the sense of the following proposition.

\begin{prop}\label{PropPSI} Let $J\in GK_{\omega}^{\mathbb{T}}(M)$ with corresponding angular coordinates $t^1,\ldots,t^m$ as in Proposition \ref{LemmeCoordonnéesCX}. Locally on $\mathring{M}$, $J$ takes the anti-diagonal form
 \begin{equation}\label{FormeBlocAntidiagonale}
\sum\limits_{i,j=1}^m\Psi_{ij}\frac{\partial}{\partial t^i}\otimes d\mu^j - \sum\limits_{i,j=1}^m\Psi^{ij}\frac{\partial}{\partial \mu^i}\otimes dt^j,
\end{equation}
where the matrix $\Psi\in C^{\infty}(\mathring{M},\mathbb{R}^{m\times m})$ is given by 
\begin{equation}\label{DéfPsi}
\frac{\partial}{\partial \mu^i} = \sum\limits_{j=1}^m \Psi_{ji}\frac{\partial}{\partial u^j},
\end{equation}
and where $\Psi^{ij}=(\Psi^{-1})_{ij}$. Let $\mathcal{K}$ denote the lagrangian distribution on $\mathring{M}$ generated by the action of $\mathbb{T}$. More generally,for an almost complex structure $\mathring{J}$ defined on $\mathring{M}$, the following statements are equivalent:
\begin{itemize}
 \item[(i)] $\mathring{J}\mathcal{K}=J\mathcal{K}$;
 \item[(ii)] $\mathring{J}$ is of the anti-diagonal form 
\begin{equation}\label{FormeBlocAntidiagonaleTilde}
\mathring{J} = \sum\limits_{i,j=1}^m\mathring{\Psi}_{ij}\frac{\partial}{\partial t^i}\otimes d\mu^j - \sum\limits_{i,j=1}^m\mathring{\Psi}^{ij}\frac{\partial}{\partial \mu^i}\otimes dt^j
\end{equation}
relative to coordinates $(\mu^j,t^j)$ induced by $J$ as in Proposition \ref{LemmeCoordonnéesCX}.
\end{itemize}
\end{prop}

\begin{proof}
We saw in the proof of Proposition \ref{LemmeCoordonnéesCX} that $\mathrm{span}(d\mu^1,\ldots,d\mu^m) = \mathrm{span}(du^1,\ldots,du^m)$. Write
$$
du^i=\sum\limits_{j=1}^m\Psi_{ij}d\mu^j,
$$
so that $Jdt^i=\sum\limits_{j=1}^m\Psi_{ij}d\mu^j$ (this equation determines $J$ entirely since $J^2=-\mathrm{Id}$). It follows that $\Psi$ verifies \eqref{DéfPsi} and $J$ is determined by
$$
J\frac{\partial}{\partial t^i}=-\sum\limits_{j=1}^m\Psi^{ji}\frac{\partial}{\partial \mu^j},
$$
which is equivalent to \eqref{FormeBlocAntidiagonale}.

Because of \eqref{DéfPsi}, we have
\begin{equation}\label{ObservationClé}
J\mathcal{K}=\mathrm{span}\left(\frac{\partial}{\partial \mu^1},\ldots,\frac{\partial}{\partial \mu^m}\right),
\end{equation}
and so an almost complex structure $\mathring{J}$ verifies (i) if and only if it takes the form
$$
\mathring{J}\frac{\partial}{\partial t^i}=-\sum\limits_{j=1}^m\mathring{\Psi}^{ji}\frac{\partial}{\partial \mu^j}
$$
for a certain matrix $\mathring{\Psi}$, which is equivalent to \eqref{FormeBlocAntidiagonaleTilde}.
\end{proof}

\begin{notation}
{\rm
We shall be interested in the almost complex structures $J\in GK_\omega^{\mathbb{T}}(M)$ whose symplectic dual $J^{*_\omega}$ is also anti-diagonal. Thus, set 
$$
DGK_{\omega}^{\mathbb{T}}(M)=\{J\in GK_\omega^{\mathbb{T}}(M)\ |\ J^{*_\omega}\mathcal{K}=J\mathcal{K}\},
$$
$$
DGAK_{\omega}^{\mathbb{T}}(M)=\{J\in GAK_\omega^{\mathbb{T}}(M)\ |\ J^{*_\omega}\mathcal{K}=J\mathcal{K}\}.
$$
}
\end{notation}

\begin{prop} \label{LemmeCoordonnéesS}
For $J\in GK_\omega^{\mathbb{T}}(M)$, the following conditions are  equivalent:
\begin{itemize}
\item[(i)] $J^{*_\omega}\mathcal{K}=J\mathcal{K}$.
\item[(ii)] The distribution $J\mathcal{K}$ is Lagrangian.
\item[(iii)] The coordinates $(\mu^j,t^j)$ are symplectic, \i.e. $(\mu^j,t^j)$ define momentum-angle coordinates on  $(M,\omega)$.
\end{itemize}
Moreover, if $J\in DGK_{\omega}^{\mathbb{T}}(M)$ and if $\mathring{J}$ is an almost complex structure on $\mathring{M}$ of anti-diagonal form 
$$
\mathring{J}=\sum\limits_{i,j=1}^m\mathring{\Psi}_{ij}\frac{\partial}{\partial t^i}\otimes d\mu^j-\sum\limits_{i,j=1}^m\mathring{\Psi}^{ij}\frac{\partial}{\partial \mu^i}\otimes dt^j
$$
with respect to momentum-angle coordinates $(\mu^j,t^j)$ induced by $J$, then $\mathring{J}^{*_\omega}$ is automatically also anti-diagonal with
\begin{equation}\label{DualDiagonal}
\mathring{J}^{*_\omega}=-\sum\limits_{i,j=1}^m\mathring{\Psi}_{ji}\frac{\partial}{\partial t^i}\otimes d\mu^j+\sum\limits_{i,j=1}^m\mathring{\Psi}^{ji}\frac{\partial}{\partial \mu^i}\otimes dt^j.
\end{equation}
\end{prop}

\begin{proof}
$(i)\Leftrightarrow (ii)$: Generally speaking, for a symplectic vector space $(V,\omega)$ equipped with a complex structure $J$ and a Lagrangian subspace $L$, the subspace $JL$ is Lagrangian if and only if $J^{*_\omega}L=JL$. Indeed, we have $\omega(JL,J^{*_\omega}L)=\omega(L,L)=0$, and so $J^{*_\omega}L\subset (JL)^{\perp_{\omega}}$. But, by definition, $JL$ is Lagrangian if and only if $JL = (JL)^{\perp_{\omega}}$. The equivalence between (i) and (ii) thus holds for any almost complex structure on $\mathring{M}$.

$(ii)\Leftrightarrow (iii)$: In general, we have
$$
\omega\left(\frac{\partial}{\partial t^i}, \frac{\partial}{\partial t^j}\right)=\omega\left(K_i,K_j\right)=0,
$$
$$
\omega\left(\frac{\partial}{\partial \mu^i},\frac{\partial}{\partial t^j}\right)=-\omega\left(K_j, \frac{\partial}{\partial \mu^i}\right)=d\mu^j\left(\frac{\partial}{\partial \mu^i}\right)=\delta^j_i.
$$
The equivalence between $(ii)$ and $(iii)$ then follows immediately from \eqref{ObservationClé}.

From the fact that the coordinates $(\mu^j,t^j)$ are symplectic, if $\mathring{J}$ is of the form \eqref{FormeBlocAntidiagonaleTilde}, we deduce formula \eqref{DualDiagonal} from $\omega(\mathring{J}\cdot,\cdot)=\omega(\cdot,\mathring{J}^{*_\omega}\cdot)$.
\end{proof}

Since a complex structure $J\in GK_\omega^{\mathbb{T}}(M)$ is compatible with $\omega$ if and only if $J=-J^{*_\omega}$ and, in this case, the condition $J^{*_\omega}\mathcal{K}=J\mathcal{K}$ is trivially satisfied, the set $DGK_\omega^{\mathbb{T}}(M)$ is an intermediate class between the Kähler structures and the generalized Kähler structures, \i.e. we have the strict inclusions
$$
K_{\omega}^{\mathbb{T}}(M)\subsetneq DGK_{\omega}^{\mathbb{T}}(M)\subsetneq GK_{\omega}^{\mathbb{T}}(M).
$$

\begin{déf}\label{DéfinitionCoordAdmissibles}
Let us call {\bf admissible coordinates} a coordinate system $(\mu^j,t^j)$ induced as in Proposition \ref{LemmeCoordonnéesCX} by an element $J\in DGK_{\omega}^{\mathbb{T}}(M)$.
\end{déf}

\begin{rem}
{\rm
\begin{itemize}
\item[(1)] In this language, Propositions \ref{PropPSI} and \ref{LemmeCoordonnéesS} imply that $J\in GK_{\omega}^{\mathbb{T}}(M)$ belongs to $DGK_{\omega}^{\mathbb{T}}(M)$ if and only if there exists admissible cooridnates $(\mu^j,t^j)$ with respect to which $J$ takes the anti-diagonal form \eqref{FormeBlocAntidiagonale}.
\item[(2)] There is a natural choice of admissible coordinates on $(M,\omega)$, obtained by taking the complex structure $J$ to be the standard Kähler structure on $M$ coming from Delzant's construction\footnote{Recall that in his famous theorem, Delzant constructs a toric symplectic manifold with prescribed moment polytope as the symplectic quotient of $\mathbb{C}^d$ by a certain sub-torus of the $\mathbb{T}^d$-action. In particular, this action preserves the standard complez structure of $\mathbb{C}^d$ and so the Kähler structure descends to the quotient (cf. for instance \cite{hitchin:4}).}. In this case, V. Guillemin \cite{guillemin:1} has found an explicit expression for the matrix $\Psi$ in terms of the fonctions $L_j$ defining the moment polytope (cf Definition \ref{DéfDelzantPolytope}):
$$
\Psi_{ij}=\frac{\partial^2}{\partial\mu^i\partial\mu^j}\left(\frac{1}{2}\sum\limits_{j=1}^mL_j\log L_j\right).
$$
\end{itemize}
}
\end{rem}

Our next theorem extends V. Guillemin's notion of symplectic potential \cite{guillemin:1, guillemin:2} of elements of $K_{\omega}(M)$ to the case of elements of $DGK_\omega^{\mathbb{T}}(M)$.

\begin{thm}\label{PropPrincipale} 
Let $\mathring{J}\in DGAK^{\mathbb{T}}_{\omega}(\mathring{M})$ of the anti-diagonal form
\begin{equation}\label{J}
\mathring{J}=\sum\limits_{i,j=1}^m\mathring{\Psi}_{ij}\frac{\partial}{\partial t^i}\otimes d\mu^j - \sum\limits_{i,j=1}^m\mathring{\Psi}^{ij}\frac{\partial}{\partial \mu^i}\otimes dt^j
\end{equation}
with respect to admissible coordinates $(\mu^j,t^j)$. Then, $\mathring{J}$ is integrable if and only if
 $\mathring{\Psi}_{ij,k}=\mathring{\Psi}_{ik,j}$ $\forall i,j,k$, whereas $\mathring{J}^{*_{\omega}}$ is integrable if and only if $\mathring{\Psi}_{ji,k}=\mathring{\Psi}_{ki,j}$. If these two conditions are met (\i.e. if $\mathring{J}$ is integrable as a generalized Kähler structure), then $\mathring{\Psi}$ is of the form
\begin{equation}\label{decA}
\mathring{\Psi}=\mathrm{Hess}(\tau)+C,
\end{equation}
where $\tau\in C^{\infty}(\mathring{\Delta})$ is strictly convex\footnote{\i.e. its Hessian is positive definite} and $C$ is a constant antisymmetric matrix. Conversely, given $\tau\in C^{\infty}(\mathring{\Delta})$ strictly convex and $C$ an antisymmetric matrix, formulas \eqref{J} and \eqref{decA} define an element $\mathring{J}$ of $DGK_{\omega}^{\mathbb{T}}(\mathring{M})$.
\end{thm}

\begin{proof}
The almost complex structure \eqref{J} is given by
\begin{equation}\label{J^*}
\mathring{J}dt^i=\sum\limits_{j=1}^m\mathring{\Psi}_{ij}d\mu^j.
\end{equation}
If $\mathring{J}$ is integrable, Proposition \ref{LemmeCoordonnéesCX} guarantees the existence of $\mathring{J}$-holomorphic coordinates $(\mathring{u}^i,\mathring{t}^i)$ such that  $(d\mathring{u}^i,d\mathring{t}^i)$ is the dual basis to $(-\mathring{J}K_i,K_i)$. Since $\mathring{J}\mathcal{K}=J\mathcal{K}$, we have $d\mathring{t}^i=dt^i$, and so equation \eqref{J^*} can be written
$$
d\mathring{u}^i=\sum\limits_{j=1}^m\mathring{\Psi}_{ij}d\mu^j.
$$
Taking the exterior derivative of this equation, we obtain the condition $\mathring{\Psi}_{ij,k}=\mathring{\Psi}_{ik,j}$ $\forall i,j,k$. Conversely, if $\mathring{\Psi}_{ij,k}=\mathring{\Psi}_{ik,j}$ $\forall i,j,k$, then taking the exterior derivative of \eqref{J^*}, we see that the 1-form $\mathring{J}^*dt^i$ is closed. It is thus locally exact which yields complex coordinates for $\mathring{J}$. We saw in Proposition \ref{LemmeCoordonnéesS} that $\mathring{J}^{*_\omega}$ takes the form \eqref{DualDiagonal}. The same argument as for $\mathring{J}$ thus shows that $\mathring{J}^{*_\omega}$ is integrable if and only if $\mathring{\Psi}_{ji,k}=\mathring{\Psi}_{jk,i}$ $\forall i,j,k$.

If $\mathring{J}$ and $\mathring{J}^{*_{\omega}}$ are integrable, then taking the sum and difference of the corresponding differential identities $\mathring{\Psi}_{ij,k}=\mathring{\Psi}_{ik,j}$ and $\mathring{\Psi}_{ji,k}=\mathring{\Psi}_{jk,i}$, we obtain the identities
\begin{equation}\label{eqAsym}
\mathring{\Psi}^s_{ij,k}=\mathring{\Psi}^s_{ik,j}, 
\end{equation}
\begin{equation}\label{eqAanti}
\mathring{\Psi}^a_{ij,k}=\mathring{\Psi}^a_{ik,j},
\end{equation}
where 
$$
\mathring{\Psi}^s = \frac{\mathring{\Psi}+\mathring{\Psi}^T}{2}, \quad \mathring{\Psi}^a=\frac{\mathring{\Psi}-\mathring{\Psi}^T}{2}
$$
are respectively the symmetric and antisymmetric parts of $\mathring{\Psi}$. Equation \eqref{eqAanti} implies that the matrix $\mathring{\Psi}^a$ is constant due to 
$$
\mathring{\Psi}^a_{ij,k}=\mathring{\Psi}^a_{ik,j}=-\mathring{\Psi}^a_{ki,j}=-\mathring{\Psi}^a_{kj,i}=\mathring{\Psi}^a_{jk,i}=-\mathring{\Psi}^a_{ij,k}.
$$
As for equation \eqref{eqAsym}, we make use of the general fact according to which a smooth $m\times m$ symmetric matrix-valued function $G$ defined on an open set $U\subset\mathbb{R}^m$ with $H^1_{dR}(U)=0$ satisfying $G_{ij,k}=G_{ik,j} \ \forall i,j,k$ is of the form $G = \mathrm{Hess}(g)$ for some function $g\in C^{\infty}(U)$. Thus, $\mathring{\Psi}^s=\mathrm{Hess}(\tau)$ for some function $\tau\in C^{\infty}(\mathring{\Delta})$. The fact that $\mathring{J}$ is $\omega$-tamed is equivalent to the positivity of $\mathring{\Psi}$, and since ${\bf x}^TC{\bf x}=0$ for all antisymmetric matrices $C$ and column vectors  ${\bf x}$, we have ${\bf x}^T\mathring{\Psi}{\bf x}={\bf x}^T\mathring{\Psi}^s{\bf x}$, from which it follows that $\tau$ is strictly convex.
\end{proof}

\begin{déf}
Given admissible coordintes $(\mu^j,t^j)$ and $\mathring{J}\in DGK^{\mathbb{T}}_{\omega}(\mathring{M})$ of antidiagonal form \eqref{J} for $\mathring{\Psi}=\mathrm{Hess}(\tau)+C$ as in the statement of Theorem \ref{PropPrincipale}, we will call the function $\tau$ the {\bf symplectic potential} of $\mathring{J}$.
\end{déf}

\subsection{Compactification and deformation}\label{Compactification et déformation}

We ask now whether a generalized almost Kähler structure $\mathring{J}\in DGAK_{\omega}^{\mathbb{T}}(\mathring{M})$ on $\mathring{M}$ is the restriction of an generalized almost Kähler structure defined on $M$? 

Let $(\mu^j,t^j)$ be admissible coordinates on $M$ and $J\in DGK^{\mathbb{T}}_\omega(M)$ (globally defined) be of the form
\begin{equation}\label{FormeJ}
J=\sum\limits_{i,j=1}^m\Psi_{ij}\frac{\partial}{\partial t^i}\otimes d\mu^j - \sum\limits_{i,j=1}^m\Psi^{ij}\frac{\partial}{\partial \mu^i}\otimes dt^j.
\end{equation}
Consider $\mathring{J}\in DGAK^{\mathbb{T}}_\omega(\mathring{M})$ (defined on $\mathring{M}$) of the form 
\begin{equation}\label{FormeJTilde}
\mathring{J}=\sum\limits_{i,j=1}^m\mathring{\Psi}_{ij}\frac{\partial}{\partial t^i}\otimes d\mu^j - \sum\limits_{i,j=1}^m\mathring{\Psi}^{ij}\frac{\partial}{\partial \mu^i}\otimes dt^j
\end{equation}
and set 
$$
\beta=\omega(\cdot,J\cdot),\quad \mathring{\beta}=\omega(\cdot,\mathring{J}\cdot).
$$
It is possible to argue as in the almost Kähler setting treated in \cite{apostolov:3} in order to obtain sufficient conditions for the compactification of such a $\mathring{J}$. Because of
\begin{equation}\label{reldtdmu}
Jdt^i=-\sum\limits_{i,j=1}^m\Psi_{ij}d\mu^j,
\end{equation}
we can write 
\begin{flalign}
\notag \mathring{\beta}-\beta &= \sum\limits_{i,j=1}^m(\mathring{\Psi}-\Psi)_{ij}d\mu^i\otimes d\mu^j+\sum\limits_{i,j,k,l=1}^m(\mathring{\Psi}^{-1}-\Psi^{-1})_{ij}(\Psi_{ik}Jd\mu^k)\otimes(\Psi_{jl}Jd\mu^l)\\
\label{MembreDeDroite}&=\sum\limits_{i,j=1}^m(\mathring{\Psi}-\Psi)_{ij}d\mu^i\otimes d\mu^j+\sum\limits_{k,l=1}^m(\Psi^T\mathring{\Psi}^{-1}\Psi-\Psi^T)_{kl}(Jd\mu^k)\otimes(Jd\mu^l).
\end{flalign}
Thus, if $\mathring{\Psi}-\Psi$ and $\Psi^T\mathring{\Psi}^{-1}\Psi-\Psi^T$ admit smooth extensions to $\Delta$, then the right hand side of \eqref{MembreDeDroite} defines a smooth $\mathbb{T}$-invariant bilinear form on the whole of $M$. It follows that $\mathring{\beta}$ (alternatively, $\mathring{J}$) admits a smooth extension to $M$. As $\mathring{M}$ is dense in $M$, by continuity, the extension verifies $\mathring{J}^2=-\mathrm{Id}$ everywhere on $M$ and is integrable provided that $\mathring{J}$ is integrable. On the other hand, a continuity argument only shows that $\mathring{\beta}$ is positive {\em semi}-definite on $M$. In order that the compactification of $\mathring{J}$ be $\omega$-tamed, we must make sure that the bilinear form
$$
\beta+\sum\limits_{i,j=1}^m(\mathring{\Psi}-\Psi)_{ij}d\mu^i\otimes d\mu^j+(\Psi^T\mathring{\Psi}^{-1}\Psi-\Psi^T)_{ij}(Jd\mu^i)\otimes(J d\mu^j)
$$
is positive definite on $M\backslash \mathring{M}$. We summarize the discussion in the following

\begin{thm}\label{ThmC123}
Let $J\in DGK_{\omega}^{\mathbb{T}}(M)$ be of the form \eqref{FormeJ} and $\mathring{J}\in DGAK^{\mathbb{T}}_\omega(\mathring{M})$ (resp. $\mathring{J}\in DGK^{\mathbb{T}}_\omega(\mathring{M})$) of the form \eqref{FormeJTilde}. If the matrix $\mathring{\Psi}$ associated with $\mathring{J}$ verifies the three conditions
\begin{itemize}
 \item[(C1)] $\mathring{\Psi}-\Psi$ admits a smooth extension to $\Delta$;
 \item[(C2)] $\Psi^T\mathring{\Psi}^{-1}\Psi-\Psi^T$ admits a smooth extension to $\Delta$;
 \item[(C3)] $\beta+\sum\limits_{i,j=1}^m(\mathring{\Psi}-\Psi)_{ij}d\mu^i\otimes d\mu^j+(\Psi^T\mathring{\Psi}^{-1}\Psi-\Psi^T)_{ij}(Jd\mu^i)\otimes(J d\mu^j)$ is positive definite on $M\backslash \mathring{M}$;
\end{itemize}
then $\mathring{J}$ is the restriction of an element of $DGAK^{\mathbb{T}}_\omega(M)$ (resp. of $DGK^{\mathbb{T}}_\omega(M)$).
\end{thm}

As in \cite{apostolov:3} (cf. Remark 4), conditions (C1), (C2) can be recasted as follows:

\begin{lemme}\label{ConditionsPrime} 
The conditions (C1), (C2) is equivalent to
\begin{itemize}
\item[(C1)] $\mathring{\Psi}-\Psi$ admits a smooth extension to $\Delta$,
\item[(C2)'] the smooth extension of $\Psi^{-1}\mathring{\Psi}$ on $\Delta$ is invertible.
\end{itemize}
\end{lemme}

As an immediate consequence of Theorem \ref{ThmC123}, we obtain

\begin{cor}\label{CorDéformation}
Let $(\mu^j,t^j)$ be admissible coordinates on $M$ and let $J_0\in K_{\omega}^{\mathbb{T}}(M)$ be an $\omega$-compatible complex structure of the form
\begin{equation*}
J_0=\sum\limits_{i,j=1}^mS_{ij}\frac{\partial}{\partial t^i}\otimes d\mu^j - \sum\limits_{i,j=1}^mS^{ij}\frac{\partial}{\partial \mu^i}\otimes dt^j
\end{equation*} 
for some positive definite symmetric matrix $S$. Consider also an arbitrary antisymmetric matrix $C$ and define a family of complex structures $\mathring{J}_t\in DGK^{\mathbb{T}}_\omega(\mathring{M})$ ($t\in\mathbb{R}$) by
\begin{equation}\label{Famille}
\mathring{J}_t=\sum\limits_{i,j=1}^m\Psi_{ij}(t)\frac{\partial}{\partial t^i}\otimes d\mu^j - \sum\limits_{i,j=1}^m\Psi^{ij}(t)\frac{\partial}{\partial \mu^i}\otimes dt^j,
\end{equation}
where
$$
\Psi(t)=S+tC.
$$
For sufficiently small values of $|t|$, the family $\mathring{J}_t$ is the restriction to $\mathring{M}$ of a family $J_t\in DGK_\omega^{\mathbb{T}}(M)$. 
\end{cor} 

\begin{proof}
By \cite{apostolov:3}, we know that conditions (C1), (C2)' and (C3) are verified for $\Psi(0)=S$. It suffices to notice that for $t$ small enough, $\Psi(t)$ continues to verify conditions (C1), (C2)', (C3) as $M$ is compact.
\end{proof}

By a  theorem of R. Goto \cite{goto:1, goto:2} (see also \cite{hitchin:5}), on a compact Kähler manifold $(M,\omega, J)$ equipped with a holomorphic Poisson bivector $\sigma\neq 0$, the trivial generalized Kähler structure $(\mathcal{J}_\omega,\mathcal{J}_J)$ can be deformed in the direction of $[\sigma\omega]\in H^{0,1}(M,T^{1,0})$\footnote{If we see  $\sigma=\sum\limits_{k,\ell}\sigma^{k\ell}\frac{\partial}{\partial z^k}\otimes\frac{\partial}{\partial z^{\ell}}$ and $\omega=\sum\limits_{k,\ell}\omega_{k\overline{\ell}}dz^k\otimes d\overline{z}^{\ell}$ as bundle maps $\Lambda^{1,0}\rightarrow T^{1,0}$ and $T^{0,1}\rightarrow \Lambda^{1,0}$ respectively, then the class $[\sigma\omega]$ is represented by $\sigma\circ\omega = \sum\limits_{j,k,\ell}\omega_{\ell \overline{j}}\sigma^{\ell k}d\overline{z}^j\otimes\frac{\partial}{\partial z^k}\in\Omega^{0,1}(M,T^{1,0})$.} into a nontrivial generalized Kähler structure $(\mathcal{J}_1(t),\mathcal{J}_2(t))$. 
More precisely, the complex structures $J_{\pm}(t)$ of the underlying hermitian structures depend analytically of $t$ and if $z^1,\ldots, z^m$ are local holomorphic coordinates for $J_{+}(0)$ with respect to which we have 
\begin{equation*}
\left.\frac{d}{dt}\right|_{t=0}J_{+}(t)\frac{\partial}{\partial \overline{z}^j}=\sum\limits_{k=1}^m\alpha_{\overline{j}k}\frac{\partial}{\partial z^k}+\beta_{\overline{j}\overline{k}}\frac{\partial}{\partial \overline{z}^k},
\end{equation*}
then the Kodaira-Spencer class of the deformation $J_+(t)$ is locally represented by the tensor $\sum\limits_{j,k=1}^m\alpha_{\overline{j}k}d\overline{z}^j\otimes\frac{\partial}{\partial z^k}$ with $\alpha_{\overline{j}k}=\sum\limits_{\ell=1}^m\omega_{\ell \overline{j}}\sigma^{\ell k}$. The first variation of $J_-(t)$ yields the opposite class.

\begin{prop}\label{propGoto}
The Kodaira-Spencer class of the deformation $J_t$ of  Corollary \ref{CorDéformation} is $[\sigma\omega]$ where $\sigma$ is the  holomorphic Poisson structure given by 
$$
\sigma = 2\sum\limits_{j,k=1}^m C_{jk}K_j^{\;1,0}\otimes K_k^{\;1,0}.
$$
\end{prop}

\begin{proof}
Let $z^j=u^j+it^j$ be the complex coordinates defined by $J_0$ as in Proposition \ref{LemmeCoordonnéesCX}. By virtue of \eqref{DéfPsi}, we can write
$$
J_t=\sum\limits_{k,\ell,p=1}^m\Psi_{k\ell}(t)S^{\ell p}\frac{\partial}{\partial t^k}\otimes du^{p} - \Psi^{k\ell}(t)S_{pk}\frac{\partial}{\partial u^p}\otimes dt^{\ell}.
$$
Using the relations 
$$
du^p\left(\frac{\partial}{\partial \overline{z}^{j}}\right)=\tfrac{1}{2}\delta_{pj}, \quad dt^{\ell}\left(\frac{\partial}{\partial \overline{z}^{j}}\right)=\tfrac{i}{2}\delta_{\ell j}
$$
as well as 
$$
(\Psi^{-1})'(0)=-S^{-1}CS^{-1},
$$
we compute
\begin{equation}\label{DéformationDePremierOrdre}
\left.\frac{d}{dt}\right|_{t=0}J_t\frac{\partial}{\partial \overline{z}^{j}}=\sum\limits_{k=1}^mi(CS^{-1})_{kj}\frac{\partial}{\partial z^k};
\end{equation}
\i.e. $\alpha_{\overline{j}k}=i(CS^{-1})_{kj}$. On the one hand, we have
$$
\omega=\frac{i}{2}\sum\limits_{k,\ell=1}^mS^{k \ell} dz^k\otimes d\overline{z}^\ell,
$$
and using the relation $\frac{\partial}{\partial z^j}=-iK_j^{\; 1,0}$, we can write locally
$$
\sigma=-2\sum\limits_{k,\ell=1}^mC_{k\ell}\frac{\partial}{\partial z^k}\otimes\frac{\partial}{\partial z^\ell}.
$$
It follows that $\sigma$ is a holomorphic Poisson structure, and
$$
\sum\limits_{\ell=1}^m\omega_{\ell\overline{j}}\sigma^{\ell k} =i(CS^{-1})_{kj}.
$$
\end{proof}

\begin{rem}\label{RemGualtieri}
{\rm
It has been observed in \cite{apostolov:4} in dimension 4 and by N. Hitchin \cite{hitchin:1} in general that for any generalized Kähler structure $(g,J_+,J_-,b)$, the bivector $P=\frac{1}{2}[J_+,J_-]g^{-1}:T^*M\rightarrow TM$ gives rise to the holomorphic Poisson structures
\begin{equation}\label{PoissonHolo2}
\sigma_\pm=\left([J_+,J_-]g^{-1}\right)^{2,0}=P-iJ_\pm P
\end{equation}
The holomorphic Poisson structure $\sigma_t$ associated with the family \eqref{Famille}
$$
\sigma_t=-4\sum\limits_{j,k=1}^m\left(tC+tS(\Psi(t)^{T})^{-1}CS^{-1}\Psi(t)^{T}\right)_{jk}\frac{\partial}{\partial z^j}\otimes\frac{\partial}{\partial z^k}.
$$
It is not difficult to see that in dimension 4, this reduces to $\sigma_t=4t\sigma$, whereas in general, we have $\sigma_t=4t\sigma+ O(t^2)$.
In the spirit of \cite{hitchin:2, gualtieri:2}, another way to produce an element of $GK_{\omega}^{\mathbb{T}}(M)$ by deformation of a toric Kähler structure $J_0$ is to start with a $\mathbb{T}$-invariant vector field $X$ and consider the family of 2-forms
$$
\omega_t=\frac{1}{t}\int_0^t(\varphi^X_s)^*\omega ds,
$$
where $\varphi^X_s$ is the flow of $X$. We have $\omega_0=\omega$ and for all $t>0$, $\omega_t$ is a closed 2-form, tamed by $J_0$ for $t$ small enough. In fact, if $X$ is holomorphic, $(\varphi^X_s)_*$ commutes with $J$, and the 2-form $\omega_t$ is symplectic for all $t>0$. For example, if $X$ is of the form $X=\sum\limits_{j=1}^mX^jK_j$, then its flow is of the form
$$
\varphi^X_s(\mu^1,\ldots,\mu^m,t^1,\ldots,t^m)=\left(\mu^1,\ldots,\mu^m,t^1+sX^1,\ldots,t^m+sX^m\right),
$$
from where we compute
$$
\omega_t=\omega+\frac{t}{2}\sum\limits_{j,k=1}^m\frac{\partial X^j}{\partial \mu^k}d\mu^j\wedge d\mu^k.
$$
While the deformation in Corollary \ref{CorDéformation} applies to any compact symplectic toric manifold, there are several results in the literature providing such deformations in special cases. For instance, Y. Lin and S. Tolman \cite{lin:1} developped a notion of symplectic reduction for generalized Kähler structures (see also \cite{hu:1, gualtieri:4}).They show that under certain conditions on the moment polytope $\Delta$, there exists a deformation of generalized Kähler structures $(\mathcal{J}_{\omega},\mathcal{J}_{\epsilon})$ (in the sense introduced by M. Gualtieri \cite{gualtieri:1}) of the standard Kähler structure $(\omega, J)$ on $\mathbb{C}^d$ which descends to the symplectic quotient to define a generalized Kähler structure $(\mathcal{J}_{\omega_{\mathrm{red}}},\hat{\mathcal{J}}_{J})$, where $\omega_{\mathrm{red}}$ is the reduced symplectic form. For instance, this applies to the Hirzebruch surfaces. Similarly, in \cite{gualtieri:4} are obtained explicit deformations of the standard Fubini-Study Kähler structure $(\omega_{FS},J)$ on $\mathbb{C}P^2$. The results in \cite{hitchin:2} give rise to toric deformations in the case of toric Fano manifolds.

}
\end{rem}

\section{The generalized Hermitian scalar curvature}\label{ChapMomap}

In this section, we compute the moment map for the action of a subgroup of $\mathrm{Ham}(M,\omega)$ on $GAK_{\omega}(M)$ and on $DGK_{\omega}(M)$ in admissible coordinates. This generalizes the formulae in \cite{abreu:2, donaldson:2, lejmi:1} for the hermitian scalar curvature of a toric almost Kähler metric and suggests a definition of a "scalar curvature" for generalized Kähler structures in $DGK_{\omega}^{\mathbb{T}}(M)$. In section \ref{SectionExtremal}, we use these definitions to introduce a natural notion of extremality in $GAK_{\omega}(M)$. In the case of $DGK_{\omega}^{\mathbb{T}}(M)$, we show that extremality is equivalent to the generalized Hermitian scalar curvature being an affine function of the moment coordinates. This generalizes an important theorem of M. Abreu \cite{abreu:1}.

Let $(M,\omega, \mathbb{T})$ be a compact symplectic toric manifold of real dimension $2m$ with moment map $\mu:M\rightarrow \Delta\subset \mathfrak{t}^*$. We adopt the viewpoint developped in Section \ref{section3} and regard elements of $GAK^{\mathbb{T}}_\omega(M)$ as $\mathbb{T}$-invariant complex endomorphisms of the complexified tangent bundle. The group $\mathrm{Ham}^{\mathbb{T}}(M,\omega)$ of $\mathbb{T}$-invariant hamiltonian diffeomorphisms acts on $GAK^{\mathbb{T}}_\omega(M)$ with infinitesimal action $V^{\sharp}$ given by \eqref{Vfondamental} for $V\in \mathfrak{ham}^{\mathbb{T}}(M,\omega)$. The Lie algebra $\mathfrak{ham}^{\mathbb{T}}(M,\omega)$ consists of the $\mathbb{T}$-invariant hamiltonian vector fields on $M$. Seen as a space of functions by means of the correspondence \eqref{ham-Fonctions}, this is simply the $\mathbb{T}$-invariant elements of $C^{\infty}_0(M)$. If $V=\mathrm{grad}_{\omega}h$ for some function $h\in C^{\infty}_0(M)^{\mathbb{T}}$, then $V^{\sharp}$ takes the following form relative to admissible coordinates $(\mu^j,t^j)$:
\begin{equation}\label{ChampInduit}
V^{\sharp}_{\;K}=\sum\limits_{j=1}^m(dh_{,j}\circ K)\otimes\frac{\partial}{\partial t^j}-dh_{,j}\otimes K\frac{\partial}{\partial t^j}.
\end{equation}

Let $C^{\infty}_{c,0}(M)^{\mathbb{T}}\subset C^{\infty}_0(M)^{\mathbb{T}}$ denote the ideal of functions with support in $\mathring{M}$ and $\mathrm{Ham}_c^{\mathbb{T}}(M,\omega)\mathrel{\unlhd}\mathrm{Ham}^{\mathbb{T}}(M,\omega)$ the corresponding connected subgroup.

\begin{thm}\label{ThmMomapActionRéduite}
The action of $\mathrm{Ham}_c^{\mathbb{T}}(M,\omega)$ on $GAK_{\omega}^{\mathbb{T}}(M)$ is hamiltonian with moment map $\nu:GAK^{\mathbb{T}}_\omega(M)\rightarrow (C^{\infty}_{c,0}(M)^{\mathbb{T}})^*$ given by
\begin{equation}\label{MomapActionRéduite}
\nu^f(K)=-\int_{\mathring{M}}f\left(\sum\limits_{i,j=1}^m\frac{\partial^2 \mathfrak{Re}Q_{ij}}{\partial \mu^i\partial\mu^j}\right)v_\omega,
\end{equation}
where $Q_{ij}=\omega(KK_i,K_j)$. For $J\in DGK_{\omega}^{\mathbb{T}}(M)$, the following alternative expression holds:
\begin{equation}\label{ExprAlt}
\nu^f(J)=\int_{\mathring{M}}f\left(\sum\limits_{i,j=1}^m\frac{\partial^2 S^{ij}}{\partial \mu^i\partial\mu^j}\right)v_\omega,
\end{equation}
where $S_{ij}=\tau_{,ij}$ for $\tau\in C^{\infty}(\mathring{\Delta})$ the symplectic potential of $J$ and $S^{ij}=(S^{-1})_{ij}$.
\end{thm}

\begin{proof}
Formula \eqref{ChampInduit} together with the fact that $M\backslash\mathring{M}=\mu^{-1}(\partial\Delta)$ has measure 0 allows us to write
\begin{flalign*}
\Omega_K(V_{\;K}^{\sharp},\dot{K}) & = \frac{1}{2}\int_M\mathrm{tr}(K\circ V_{\;K}^{\sharp}\circ \dot{K})v_\omega \\
& = \sum\limits_{j=1}^m\int_{\mathring{M}}\mathrm{tr}\left((df_{,j}\circ \dot{K})\otimes \frac{\partial}{\partial t^j}\right)v_\omega\\
& = \sum\limits_{j=1}^m\int_{\mathring{M}}df_{,j}\left(\dot{K}\frac{\partial}{\partial t^j}\right)v_\omega.
\end{flalign*}
Since $\mathbb{T}$ acts freely on $\mathring{M}$ (with $\mathring{\Delta}$ identified with the orbit space), $\mu:\mathring{M}\rightarrow \mathring{\Delta}$ defines a trivial principal torus bundle: $\mathring{M}\cong\mathring{\Delta}\times \mathbb{T}$. We have $v_{\omega}=(-1)^{m-1}dx^1\wedge\ldots\wedge dx^m\wedge dt^1\wedge\ldots\wedge dt^m$ so if we set $C_m=\int_{\mathbb{T}}(-1)^{m-1}dt^1\wedge\ldots\wedge dt^m$, we can write  
\begin{equation}\label{FormeTemporaire}
\Omega_K(V_{\;K}^{\sharp},\dot{K}) = C_m\sum\limits_{j=1}^m\int_{\mathring{\Delta}}df_{,j}\left(\dot{K}\frac{\partial}{\partial t^j}\right)v_0,
\end{equation}
where $v_0=dx^1\wedge\ldots\wedge dx^m$. 
If the matrix representation of $K$ relative to the basis $\left(\frac{\partial }{\partial \mu},\frac{\partial }{\partial t}\right)$ de $T^{\mathbb{C}}\mathring{M}$ is
$$
K=
\left(
\begin{array}{cc}
P & Q \\
R & S
\end{array}
\right),
$$
then
\eqref{FormeTemporaire} takes the form
\begin{equation}
\Omega_K(V^{\sharp}_K,\dot{K})=C_m\sum\limits_{i,j=1}^m\int_{\mathring{\Delta}}f_{,ij}\dot{Q}_{ij}v_0.
\end{equation}
This computation suggests that the moment map is
\begin{equation}\label{FormuleNu}
\nu^f(K)=-C_m\sum\limits_{i,j=1}^m\int_{\mathring{\Delta}}f_{,ij}Q_{ij}v_0.
\end{equation}
Here, we observe that the functions $Q_{ij}$ are well-defined and smooth on $\Delta$ since we can write $Q_{ij}=\omega(KK_j,K_i)$ which is a smooth and $\mathbb{T}$-invariant function on $M$. Consequently, if $f$ has support in $\mathring{\Delta}$, a double integration by parts allows us to shift the derivatives over to $Q_{ij}$, and thus
\begin{flalign*}
\nu^f(K) & =-C_m\sum\limits_{i,j=1}^m\int_{\mathring{\Delta}}fQ_{ij,ij}v_0\\
& = -\sum\limits_{i,j=1}^m\int_{\mathring{M}}fQ_{ij,ij}v_\omega.
\end{flalign*}
It remains to check the equivariance of $\nu$, namely the relation $\nu^{\varphi\cdot f}(\varphi\cdot K)=\nu^f(K)$ for $\varphi\in \mathrm{Ham}^{\mathbb{T}}(M,\omega)$. Let $\varphi$ be the flow at time 1 of $\mathrm{grad}_{\omega}h$ for $h\in C^{\infty}_0(M)^{\mathbb{T}}$. As in Remark \ref{RemGualtieri} (2), we compute
$$
\varphi_*=
\left(
\begin{array}{cc}
I & 0 \\
(h_{,ij}) & I
\end{array}
\right),
$$
so in particular, $\varphi$ preserves the fields vector $K_i$. It follows that $\varphi$ acts on $K$ by changing $Q_{ij}$ to $\omega(\varphi_*K\varphi^{-1}_*K_j,K_i)=Q_{ij}\circ \varphi^{-1}=\varphi\cdot Q_{ij}$.
Next, using the naturality of the Lie derivative on $(\varphi\cdot Q_{ij})_{,ij}=\mathcal{L}_{\partial\slash\partial \mu^i}\mathcal{L}_{\partial\slash\partial \mu^j}(\varphi\cdot Q_{ij})$, we get
$$
(\varphi\cdot Q_{ij})_{,ij} = \varphi\cdot \mathcal{L}_{\varphi^{-1}_*\partial\slash\partial \mu^i}\mathcal{L}_{\varphi^{-1}_*\partial\slash\partial \mu^j}Q_{ij},
$$
where
$$
\varphi^{-1}_*\frac{\partial}{\partial \mu^i} = \frac{\partial}{\partial \mu^i}-\sum\limits_{k=1}^{m}h_{,ki}\frac{\partial}{\partial t^k}.
$$
But the functions $Q_{ij}$ are $\mathbb{T}$-invariant, so $\mathcal{L}_{\partial\slash\partial t^k}Q_{ij}=0$ and it remains
$$
(\varphi\cdot Q_{ij})_{,ij} = \varphi\cdot (Q_{ij,ij}).
$$
Since $\varphi$ preserves the symplectic volume form $v_\omega$, we have
$$
\nu^{\varphi\cdot f}(\varphi\cdot K)=-\sum\limits_{i,j=1}^m\int_{\mathring{M}}\varphi\cdot (fQ_{ij,ij}v_{\omega})=\nu^{f}(K).
$$
Finally, note that the expression $\sum\limits_{i,i=1}^m Q_{ij,ij}$ is {\em real}. This follows from the fact that the imaginary part of $K$ is $\omega$-self-dual (cf. \eqref{conditionsAB}), and so the imaginary part of $Q_{ij}=\omega(KK_j,K_i)$ is antisymmetric. We thus get \eqref{MomapActionRéduite}.

To obtain \eqref{ExprAlt}, recall equation \eqref{J-AB} to obtain the expression $\mathfrak{Re}Q_{ij}=-\omega((J^a)^{-1}K_j,K_i)$. In terms of the identification $K\sim(S,C)$ of Theorem \ref{PropPrincipale}, we obtain $\mathfrak{Re}Q_{ij}=-S^{ij}$. 
\end{proof}

Comparing the results of Theorem \ref{ThmMomapActionRéduite} with \eqref{ThmDonaldson}, we are naturally led to the following definition.

\begin{déf}\label{DéfCourburesGénéralisées}
The {\bf generalized Hermitian scalar curvature} of $K\in GAK_{\omega}^{\mathbb{T}}(M)$ is
\begin{equation}\label{ScalGK}
u_{GK}(K)=\sum\limits_{i,j=1}^m\frac{\partial^2 \mathfrak{Re}Q_{ij}}{\partial x^i\partial x^j},
\end{equation}
where $Q_{ij}=\omega(KK_i,K_j)$.
\end{déf}

\begin{rem}
{\rm
In terms of the characterisation of Theorem \ref{PropPrincipale}, the Hermitian scalar curvature of an element $\mathcal{J}$ of $DGK_{\omega}^{\mathbb{T}}(\mathring{M})$ is of the form
\begin{equation}\label{ScalGKPot}
u_{GK}(\mathcal{J})=-\sum\limits_{i,j=1}^m \frac{\partial^2\tau^{ij}}{\partial \mu^i\partial\mu^j},
\end{equation}
where $\tau^{ij} = (\mathrm{Hess}(\tau)^{-1})_{ij}$.
}
\end{rem}


\begin{rem}
{\rm
\begin{itemize}
 \item[(1)] The function $u_{GK}(K)$ is well-defined globally, since $Q_{ij}=\omega(KK_i,K_j)\in C^{\infty}(M)^{\mathbb{T}}\cong C^{\infty}(\Delta)$. However, at present, we can only be sure that $S^{ij}$ defines an element of $C^{\infty}(M)^{\mathbb{T}}$ in dimension 4 (cf. section \ref{Compactification en dimension 4}).
 \item[(2)] When $J\in K_{\omega}^{\mathbb{T}}(M)$, formula \eqref{ScalGKPot} reduces to the formula found by M. Abreu \cite{abreu:1} for the Riemannian scalar curvature. Similarly, when $K\in AK^{\mathbb{T}}_\omega(M)$, formula \eqref{ScalGK} reduces to the formula found by S. K. Donaldson \cite{donaldson:2} and more generally, by M. Lejmi \cite{lejmi:1} for the Hermitian scalar curvature.
\end{itemize}
}
\end{rem}

\section{Extremal generalized Kähler structures}\label{SectionExtremal}

Let $(M,\omega, \mathbb{T})$ be a compact symplectic toric manifold of real dimension $2m$ with moment map $\mu:M\rightarrow \Delta\subset \mathfrak{t}^*$. It is clear that the moment map in equation \eqref{ThmDonaldson} can be replaced by $\nu^f(J)=-\int_M f(u_J-\overline{u}_J)v_{\omega}$ (for $\overline{u}_J=\int_Mu_Jv_{\omega}$) so that with respect to the identifications discussed in Section \ref{section3}, $\nu$ can be seen as the map $J\mapsto -u_J+\overline{u}_J\in C^{\infty}_0(M)$. A simple computation reveals that the critical points of $\Vert\nu\Vert^2:K_{\omega}(M)\rightarrow\mathbb{R}$ are precisely the extremal Kähler metrics in the sense of E. Calabi \cite{calabi:1}. Indeed, we have
$$
d(\Vert\nu\Vert^2)_J(\dot{J}) = 2(\nu(J),d\nu_J(\dot{J}))=2\Omega_J((\mathrm{grad}_{\omega}u_J)^{\sharp}, \dot{J})=-2\Omega_J(\mathcal{L}_{\mathrm{grad}_{\omega}u_J}J,\dot{J}),
$$
Thus, $J$ is a critical point if and only if $\mathcal{L}_{\mathrm{grad}_{\omega}u_J}J=0$. Since $J$ is $\omega$-compatible, this is equivalent to saying that $\mathrm{grad}_{\omega}u_J$ is Killing. But as is well known \cite{calabi:1}, for fixed $J$ this condition also characterizes the Kähler metrics in a given DeRham class $a\in H^2_{dR}(M)$ which are critical points of the Calabi functional $g\mapsto \int_M u_g^{\;2}v_g$. More generally, the calculation above holds true on $GAK_{\omega}^{\mathbb{T}}(M)$ provided that $\nu$ is replaced with the moment map from Theorem \ref{ThmMomapActionRéduite}. In light of this, the following definition is natural.

\begin{déf}
Let $(M,\omega,\mathbb{T},\mu)$ be a compact symplectic toric manifold. An element $K\in GAK_{\omega}^{\mathbb{T}}(M)$ is called {\bf extremal} if it is a critical point of the functional $K\mapsto\int_M (u_{GK}(K)-\overline{u_{GK}}(K))^2v_\omega$, where $\overline{u_{GK}}(K) = \int_Mu_{GK}(K)v_\omega$. An equivalent condition is
$$
\mathcal{L}_{\mathrm{grad}_{\omega}u_{GK}(K)}K=0.
$$
\end{déf}

M. Abreu has observed \cite{abreu:1} that the toric Kähler metrics which are extremal are precisely those whose scalar curvature depends in an affine manner upon the moment coordinates of Proposition \ref{LemmeCoordonnéesCX}. This characterization admits a natural extension to $DGK_{\omega}^{\mathbb{T}}(M)$:

\begin{prop}\label{CaracterisationsExtremal}
For $J\in DGK_{\omega}^{\mathbb{T}}(M)$, the following statements are equivalent.
\begin{itemize}
\item[(1)] $J$ is extremal.
\item[(2)] $\mathcal{L}_{\mathrm{grad}_{\omega}u_{GK}(J)}J=0$.
\item[(3)] The vector field $\mathrm{grad}_{\omega}u_{GK}(J)$ is Killing with respect to $g=\omega(\cdot,J\cdot)^{s}$ and also preserves the 2-form $b=-\omega(\cdot,J\cdot)^{a}$.
\item[(4)] $u_{GK}(J)$ is an affine function in the momentum variables $(\mu^1,\ldots,\mu^m)$.
\end{itemize}
\end{prop}

\begin{proof}
Let $K=A+iB$ be the endomorphism of $T^{\mathbb{C}}M$ corresponding to $J$ as in section \ref{section3} and let $X$ be a vector field on $M$. The equation $\mathcal{L}_X K=0$ is equivalent to $\mathcal{L}_XA=\mathcal{L}_XB=0$. According to \eqref{J-AB}, we have 
\begin{flalign*}
\mathcal{L}_XA=0 & \Leftrightarrow (J-J^{*_{\omega}})^{-1}(\mathcal{L}_X J-(\mathcal{L}_XJ)^{*_{\omega}})(J-J^{*_{\omega}})^{-1}=0 \\
& \Leftrightarrow\mathcal{L}_X J=(\mathcal{L}_XJ)^{*_{\omega}},
\end{flalign*}
and since $B=\frac{1}{2}(J+J^{*_{\omega}})A$, we see that under the hypothesis $\mathcal{L}_XA=0$, we have
$$
\mathcal{L}_XB=0 \Leftrightarrow \mathcal{L}_X J=-(\mathcal{L}_XJ)^{*_{\omega}}.
$$
Hence, $\mathcal{L}_X K=0$ is equivalent to $\mathcal{L}_X J=0$. Taking the Lie derivative of the equation $\omega(\cdot,J\cdot)=g-b$, we obtain $\mathcal{L}_X\omega(\cdot,J\cdot) + \omega(\cdot,\mathcal{L}_XJ\cdot)=\mathcal{L}_Xg-\mathcal{L}_Xb$. If $X=\mathrm{grad}_{\omega}u_{GK}(J)$, the first term vanishes and we see that 
$$
\mathcal{L}_{\mathrm{grad}_{\omega}u_{GK}(J)}J=0\Leftrightarrow \mathcal{L}_{\mathrm{grad}_{\omega}u_{GK}(J)}g=\mathcal{L}_{\mathrm{grad}_{\omega}u_{GK}(J)}b=0.
$$
This proves that (1) and (2) are equivalent. Statements (2) and (3) are equivalent because $\mathrm{grad}_{\omega}u_{GK}(J)$ preserves $\omega$. Assume (4) holds, so that $u_{GK}(J)=\sum\limits_{j=1}^ma_j\mu^j+b$ for certain numbers $a_1,\ldots,a_m,b\in\mathbb{R}$. Then, $du_{GK}(J)=\sum\limits_{j=1}^ma_jd\mu^j$ and so $\mathrm{grad}_{\omega}u_{GK}(J) = \sum\limits_{j=1}^ma_jK_j$. Since $J$ is $\mathbb{T}$-invariant, we have $\mathcal{L}_{K_j}J=0$ $\forall j$, whence we see that (2) holds. Finally, let us show that (3) implies (4). 
Set 
$$
V:=\mathrm{grad}_{\omega}u_{GK}(J)=\sum\limits_{j=1}^m\frac{\partial u_{GK}(J)}{\partial \mu^j}K_j.
$$
The fact that $V$ is a Killing vector field means that the tensor 
$$
DV^{\flat} = \sum\limits_{j,k=1}^m \frac{\partial^2 u_{GK}(J)}{\partial \mu^j\partial \mu^k}d\mu^k\otimes K_j^{\flat} + \sum\limits_j\frac{\partial u_{GK}(J)}{\partial \mu^j}DK_j^{\flat}
$$
is antisymmetric. Since the vector fields $K_j$ are themselves Killing, this boils down to the first term of the right hand side being antisymmetric. We have $K_j^{\flat}=\sum\limits_{\ell=1}^m(\Psi^{-1})^s_{\;\;j\ell}dt^{\ell}$, so
$$
\sum\limits_{j,k=1}^m \frac{\partial^2 u_{GK}(J)}{\partial \mu^j\partial \mu^k}d\mu^k\otimes K_j^{\flat} = \sum\limits_{k,\ell=1}^m(\mathrm{Hess}(u_{GK}(J))^T(\Psi^{-1})^s)_{k\ell}d\mu^k\otimes dt^{\ell},
$$
which implies $\mathrm{Hess}(u_{GK}(J))=0$.
\end{proof}

\begin{cor}\label{JExtrémaleDim4}
Let $(M,\omega,\mathbb{T},\mu)$ be a compact symplectic toric manifold. If there exists an extremal element $J_0\in K_{\omega}^{\mathbb{T}}(M)$, then there exists extremal elements $J\in DGK_{\omega}^{\mathbb{T}}(M)\backslash K_{\omega}^{\mathbb{T}}(M)$.
\end{cor}

\begin{proof}
According to a result of M. Abreu (\cite{abreu:2} Theorem 4.1), $J_0\in K_{\omega}^{\mathbb{T}}(M)$ is extremal if and only if $s_{J_0}$ is an affine function of $\mu^1,\ldots\mu^m$. Let $J_t$ be the deformation of $J_0$ from Corollary \ref{CorDéformation} associated with an arbitrary nonzero antisymmetric matrix $C$. For $t$ sufficiently small, $J_t\in DGK_{\omega}^{\mathbb{T}}(M)\backslash K_{\omega}^{\mathbb{T}}(M)$, and $u_{GK}(J_t)=s_{J_0}$. We obtain the desired conclusion by combining Abreu's characterization with our Proposition \ref{CaracterisationsExtremal}.
\end{proof}

\begin{rem}
{\rm
In a series of articles (see e.g. \cite{donaldson:2,donaldson:3}), S. K. Donaldson has developped a general program of characterizing compact toric varieties admitting extremal Kähler metrics in terms of suitable notions of stability for the corresponding Delzant polytopes.
}
\end{rem}

\section{The 4-dimensional case}
In this section we focus on the 4-dimensional case. In section \ref{section2}, we formulate some lemmas that will be useful later and we show that, on compact 4-manifolds, the generalized Kähler structures of symplectic type are, up to isomorphism, precisely those whose underlying complex structures induce the same orientation. In section \ref{Compactification en dimension 4} we argue that, in dimension 4, the sufficient conditions of Theorem \ref{ThmC123} are also necessary. We do so by formulating an equivalent set of conditions as is done in the Kähler setting of \cite{apostolov:3}. Finally, we provide a closed formula for the generalized Hermitian scalar curvature of elements in $DGK_{\omega}^{\mathbb{T}}(M)$ in terms of the underlying bi-Hermitian structure.

\subsection{Generalized Kähler structures of symplectic type in dimension 4}\label{section2}
In this subsection, $M$ denotes a smooth manifold of dimension 4. In this case, the underlying complex structures $J_{\pm}$ of a generalized Kähler structure of symplectic type $(J_+,J_-,g,b)$ induce the same orientation \cite{gualtieri:1}. In particular, $(J_+,J_-,g)$ forms a {\em bi-Hermitian} structure in the sense of \cite{apostolov:4} and we have \cite{pontecorvo:1}:

\begin{lemme}\label{LemmePontecorvo}
If $(J_+,J_-,g,b)$ is a generalized Kähler structure with $J_+$ and $J_-$ inducing the same orientation on, then 
\begin{equation}\label{IdentitéCruciale}
J_+J_-+J_-J_+ = -2p\mathrm{Id},
\end{equation}
where $p=-\frac{1}{4}\mathrm{tr}(J+J_-)\in [-1,1]$ is called the {\em angle function}. Moreover, $p=\pm 1$ if and only if $J_+=\pm J_-$.
\end{lemme}

Recall that the {\em Lee form} of an almost Hermitian metric $(g,J)$ with fundamental form $F=g(J\cdot,\cdot)$ is the 1-form $\theta=J\delta F$, also characterized as the unique 1-form such that $dF=\theta\wedge F$. A Hermitian metric is called {\em Gauduchon} \cite{gauduchon:3} if $\delta\theta=0$.

\begin{lemme}[\cite{apostolov:2}]\label{GK-Gauduchon} 
If $(J_+,J_-,g,b)$ is a generalized Kähler structure with $J_+$ and $J_-$ inducing the same orientation, then the metric $g$ is Gauduchon with respect to $J_+$ and $J_-$, and the Lee forms are related by 
$$
\theta_++\theta_-=0, \quad \theta_+=*db.
$$
Here, $*$ is the Hodge operator relative to $g$ and the orientation induced by $J_{\pm}$.
\end{lemme}


Recall from \cite{gualtieri:1} that the bundle isomorphisms of $TM\oplus T^*M$ that preserve both the natural inner product and the Courant bracket (called {\em Courant isomorphisms}) are of the form $f_*\circ e^b$ for $f\in \mathrm{Diff}(M)$, $b\in \Omega^2(M)$ a closed 2-form, and $e^b:X\oplus\xi\mapsto X\oplus(b(X,\cdot)+\xi)$.

\begin{thm}[\cite{hitchin:2}]
Let $M$ be a compact 4-dimensional. A generalized Kähler structure $(\mathcal{J}_1,\mathcal{J}_2)$ on $M$ is Courant equivalent to a generalized Kähler structure of symplectic type if and only if the complex structures $J_+$ and $J_-$ induce the same orientation.
\end{thm} 

\begin{proof}
If $(\mathcal{J}_1,\mathcal{J}_2)$ is Courant equivalent to a generalized Kähler structure of symplectic type, then $J_+$ and $J_-$ induce the same orientation \cite{gualtieri:1}. For the converse, assume $J_+$ and $J_-$ induce the same orientation. Without loss of generality, we may assume that $J_+\neq\pm J_-$. In this case, the first Betti number of $M$ is even \cite{apostolov:2} and we face the following alternative \cite{apostolov:4}:
\begin{itemize}
\item[(I)] $J_+(x)\neq J_-(x)$ $\forall x\in M$,
\item[(II)] $J_+(x)\neq -J_-(x)$ $\forall x\in M$.
\end{itemize}
Assume (I) holds, \i.e. $p(x)<1$ $\forall x\in M$, where $p$ is the angle function introduced in Lemma \ref{LemmePontecorvo}. Consider the 2-form 
$$
\omega=F_+-\frac{1}{2(1-p)}g[J_+,J_-]J_+,
$$
where $F_+=gJ_+$ is the fundamental form of the Hermitian structure $(g,J_+)$. This form is globally defined on $M$ and its codifferential was computed in the proof of Proposition 4 of \cite{apostolov:4} to be $\delta\omega=-\frac{1}{2}\omega(\theta_++\theta_-)^\sharp$. However, Lemma \ref{GK-Gauduchon} implies that $\omega$ is co-closed. Since $d=-*\delta*$ in dimension 4 and $\omega$ is self-dual, we see that $\omega$ is closed. The symmetric part of $\omega(\cdot,J_+\cdot)$ being $g$, it follows that $\omega$ is symplectic. To conclude, it suffices to check that $g=-\tfrac{1}{2}\omega(J_+-J_-)$. Indeed, it will then follow from Proposition \ref{PropEnrietti} that for $b_\omega=-\tfrac{1}{2}\omega(J_++J_-)$, the generalized Kähler structure corresponding to $(J_+,J_-,g,b_\omega)$ is of the form $(\mathcal{J}_{\omega}, \mathcal{J})$. In particular, $e^{b_\omega-b}\cdot(\mathcal{J}_1,\mathcal{J}_2)=(\mathcal{J}_{\omega},\mathcal{J})$. We have
$$
\omega(J_+-J_-)=-g-gJ_+J_--\frac{1}{2(1-p)}g[J_+,J_-]J_+(J_+-J_-).
$$
Using that $g([J_+,J_-]\cdot,\cdot)$ is $J_+$-anti-invariant along with the identity $(J_+-J_-)^2 = -2(1-p)\mathrm{Id}$, we may write
\begin{flalign*}
g([J_+,J_-]J_+(J_+-J_-)\cdot,\cdot)&=2(1-p)g(J_++J_-)\cdot,J_+\cdot)\\
&=2(1-p)(g-g(J_+J_-\cdot,\cdot)).
\end{flalign*} 
We see then that $\omega(J_+-J_-)=-2g$.

If (II) holds, consider $J_-'=-J_-$ so that $(J_+,J'_-)$ satisfies (I) and so $(\mathcal{J}_1,\mathcal{J}_2)$ is Courant equivalent to $(\mathcal{J},\mathcal{J}_{\omega})$.
\end{proof}

\subsection{Compactification in dimension 4}\label{Compactification en dimension 4}

Unless stated otherwise, we assume in this section that $(M,\omega, \mathbb{T})$ is a compact symplectic toric manifold of real dimension $4$ with moment map $\mu:M\rightarrow \Delta\subset \mathfrak{t}^*$ and consider $\mathring{J}\in DGK^{\mathbb{T}}_\omega(\mathring{M})$ of the form \eqref{FormeBlocAntidiagonale} with respect to admissible coordinates $(\mu^j,t^j)$ on $\mathring{M}$. We begin by writing down some identities valid in dimension 4. According to Theorem \ref{PropPrincipale}, the decomposition $\mathring{\Psi}=\mathring{\Psi}^s+\mathring{\Psi}^a$ of $\mathring{\Psi}$ into its symmetric and antisymmetric parts is of the form $\mathring{\Psi}^s=S$, $\mathring{\Psi}^a=C$ for some positive definite symmetric matrix $S$ and a constant antisymmetric matrix $C=\bigl(\begin{smallmatrix}
0&c\\ -c&0
\end{smallmatrix} \bigr)$. Therefor, we have the decomposition $\mathring{\Psi}^{-1}=(\mathring{\Psi}^{-1})^s+(\mathring{\Psi}^{-1})^a$ with
\begin{equation}\label{partiesSymAsym}
(\mathring{\Psi}^{-1})^s = \frac{\det S}{\det \mathring{\Psi}}S^{-1}, \quad (\mathring{\Psi}^{-1})^a=-\frac{1}{\det \mathring{\Psi}}C.
\end{equation}
Also, the Riemannian metric $\mathring{g}=\omega(\cdot,\mathring{\Psi}\cdot)^s$ and the 2-form $\mathring{b}=-\omega(\cdot,\mathring{\Psi}\cdot)^a$ are given by
\begin{equation}\label{gDiagoDim4}
\mathring{g}=\sum\limits_{i,j=1}^2 S_{ij}d\mu^i\otimes d \mu^j+\frac{\det S}{\det \mathring{\Psi}}S^{ij} dt^i\otimes dt^j,
\end{equation}
\begin{equation}\label{bdim4}
\mathring{b}=-cd\mu^1\wedge d\mu^2+\frac{c}{\det\mathring{\Psi}}dt^1\wedge dt^2.
\end{equation}
The angle function $\mathring{p}=-\frac{1}{4}\mathrm{tr}(\mathring{J}\mathring{J}^{*_\omega})$ from Lemma \ref{LemmePontecorvo} is given by
\begin{equation}
\mathring{p}=\frac{c^2-\det S}{\det\mathring{\Psi}}.
\end{equation}
In particular, 
\begin{equation}\label{Identitésp}
\frac{1-\mathring{p}}{2}=\frac{\det S}{\det\mathring{\Psi}}, \quad \frac{1+\mathring{p}}{2}=\frac{c^2}{\det\mathring{\Psi}}.
\end{equation}
Finally, the determinants are related by the formula
\begin{equation}\label{RelationDet}
\det{\mathring{\Psi}} = \det S + c^2.
\end{equation}

\begin{thm} \label{PropCompactificationDim4}
Consider $J \in DGK^{\mathbb{T}}_\omega(M)$ of the form \eqref{FormeJ} and $\mathring{J}\in DGK^{\mathbb{T}}_\omega(\mathring{M})$ of the form \eqref{FormeJTilde} with respect to admissible coordinates $(\mu^j,t^j)$. If $\mathring{J}$ satisfies conditions (C1) and (C2) of Theorem \ref{ThmC123} relative to $J$, then $\mathring{J}$ is the restriction of an element of $DGK^{\mathbb{T}}_\omega(M)$.
\end{thm}

\begin{proof}
By the arguments of section \ref{Compactification et déformation}, $\mathring{J}$ is the restriction of a complex structure $\mathring{J}_c$ on $M$. It remains to show that the (non-degenerate) bilinear form $\mathring{\beta}_c=\omega(\cdot,\mathring{J}_c\cdot)$ is positive definite on $M\backslash \mathring{M}$. By continuity, we know that $\mathring{\beta}_c$ is positive semi-definite there. Consequently, $\mathring{\beta}_c$ will be positive definite provided that the antisymmetric part $-\mathring{b}_c$ of $\mathring{\beta}_c$ vanishes on $M\backslash \mathring{M}$. Using \eqref{reldtdmu}, equation \eqref{bdim4} can alternatively be writen
$$
\mathring{b}=-c\left(\mathrm{Id} - \frac{\det\Psi}{\det \mathring{\Psi}}J^*\right)d\mu^1\wedge d\mu^2.
$$
By continuity, this formula holds true everywhere. Indeed, the 1-forms $d\mu^i$ are globally defined as is the quotient $\frac{\det\Psi}{\det \mathring{\Psi}}$ (cf. condition (C2)' of Lemma \ref{ConditionsPrime}). However, $d\mu^1\wedge d\mu^2$ vanishes on $M\backslash \mathring{M}$ since $d\mu^i$ is $\omega$-dual to $K_i$ and the $K_i$'s are linearly dependent on $M\backslash \mathring{M}$.
\end{proof}

\begin{prop}\label{CoroM^o}
Consider $J \in DGK^{\mathbb{T}}_\omega(M)\backslash K_{\omega}^{\mathbb{T}}(M)$. Then 
\begin{flalign*}
M\backslash \mathring{M} & = \{x\in M \ \vert \ \text{$J(x)$ is compatible with $\omega(x)$}\}\\
& = \{x\in M \ \vert \ J(x) = -J^{*_\omega}(x)\},\\
& = \{x\in M \ \vert \ p(x) = -1\},
\end{flalign*}
where $p=-\frac{1}{4}\mathrm{tr}(JJ^{*_\omega})$.
\end{prop}

\begin{proof}
Assume $J$ is of the form \eqref{FormeJ} relative to admissible coordinates. Combining \eqref{Identitésp} and \eqref{RelationDet}, we obtain
\begin{equation}\label{p-detS}
 \frac{1-p}{2}=\frac{\det S}{\det S+c^2},
\end{equation}
from where $p(x)>-1$ $\forall x\in\mathring{M}$. Moreover, since $\beta=\omega(\cdot,J\cdot)$ takes the form
$$
\beta = \sum\limits_{i,j=1}^2 \Psi_{ij} d\mu^i\otimes d \mu^j+\Psi^{ij} dt^i\otimes dt^j,
$$
we may write $\Psi^{ij} = \beta(K_i,K_j)$, where $\beta(K_i,K_j)$ is a smooth function defined on the whole of $M$. It follows that $\Psi^{-1}\in C^{\infty}(\Delta)$. Moreover, $\det(\Psi^{-1})=0$ on $\partial\Delta$ since the vector fields $K_i$, $i=1,2$ are linearly dependent on $M\backslash\mathring{M}$. By \eqref{RelationDet}, this implies that $\det S\rightarrow +\infty$ when $x\rightarrow \partial\Delta$. Taking the limit in \eqref{p-detS}, this implies in turn that $\frac{1-p}{2}=1$ on $M\backslash\mathring{M}$; \i.e. $p(x)=-1$ $\forall x\in M\backslash\mathring{M}$. The equivalence between the various expressions of $M\backslash\mathring{M}$ correspond to the fact that $p(x)=-1$ if and only if $J=-J^{*_\omega}$.
\end{proof}

\begin{cor}\label{ThmDGKtoK}
Consider $J\in DGK_\omega^{\mathbb{T}}(M)$ and $g=\omega(\cdot,J\cdot)^{s}$. Then, the metric
$$
g_{AK}:=\sqrt{\frac{1-p}{2}}g
$$
is smooth, $\mathbb{T}$-invariant and $\omega$-compatible. 
\end{cor}

\begin{proof}
Smoothness follows from the fact that $1-p$ vanishes nowhere on $M$. This is so because $p(x)=1$ if and only if $J(x)=J^{*_\omega}(x)$ and this happens at no point of $M$ since $J$ is tamed by $\omega$. Using $\eqref{Identitésp}$, it is trivial to check that $g_{AK}$ is $\omega$-compatible on $\mathring{M}$, and hence on $M$ by continuity.  
\end{proof}

In terms of admissible coordinates $(\mu^j,t^j)$ on $\mathring{M}$, we have
$$
g_{AK}=\sum\limits_{i,j=1}^2 \sqrt{\frac{1-p}{2}}S_{ij}d\mu^i\otimes d \mu^j+\left(\sqrt{\frac{1-p}{2}}\right)^{-1}S^{ij} dt^i\otimes dt^j.
$$
More generally, it is not hard to see that for any positive $f\in C^{\infty}(\mathring{M})$, the metric defined on $\mathring{M}$ by
\begin{equation}\label{Défg_f}
\mathring{g}_f=\sum\limits_{i,j=1}^2 fS_{ij}d\mu^i\otimes d \mu^j+f^{-1}S^{ij} dt^i\otimes dt^j,
\end{equation}
is compatible with $\omega$. In particular, for $f\equiv 1$, the resulting metric is integrable (Theorem \ref{PropPrincipale}). For this reason, we introduce the following notation:

\begin{notation}
{\rm
Denote by $\mathring{g}_K$ the toric Kähler metric on $\mathring{M}$ corresponding to the function $f\equiv 1$. In other words,
\begin{equation}\label{Défg_K}
\mathring{g}_K = \sum\limits_{i,j=1}^2 S_{ij}d\mu^i\otimes d \mu^j+S^{ij} dt^i\otimes dt^j.
\end{equation}
}
\end{notation}

\begin{lemme}\label{Lemmeg_f}
Let $f\in C^{\infty}(M)^{\mathbb{T}}$ be a positive and $\mathbb{T}$-invariant function such that $f|_{M\backslash{\mathring{M}}}\equiv 1$. Then, $\mathring{g}_f$ is the restriction to $\mathring{M}$ of a toric almost Kähler metric defined on $M$ if and only if $\mathring{g}_K$ is the restriction to $\mathring{M}$ of a toric Kähler metric  defined on $M$.
\end{lemme}

\begin{rem}\label{Remg_f}
{\rm
\begin{itemize}
\item[(1)] In particular, since $\mathring{g}_{AK} = \mathring{g}_{f}$ for $f=\sqrt{\frac{1-p}{2}}$, it satisfies the hypotheses of Lemma \ref{Lemmeg_f}. 
\item[(2)] It is known since \cite{apostolov:2} that  every 4-manifold admitting a generalized Kähler structure is Kählerian. Our construction associates in a canonical way a Kähler structure (the metric $g_K$) to any element of $DGK_{\omega}^{\mathbb{T}}(M)$.
\end{itemize}
}
\end{rem}

The proof of Lemma \ref{Lemmeg_f} relies on the compactification criterion for toric almost Kähler metrics of Apostolov-Calderbank-Gauduchon-Tønnesen-Friedman \cite{apostolov:3} which we reproduce here in a form adapted to our needs.

\begin{déf}\label{DéfCoordAdaptées}
Let $(\Delta,\Lambda,\nu_1,\ldots,\nu_d)$ be a Delzant polytope (cf. Definition \ref{DéfDelzantPolytope}) and let $x_0$ be a point in the interior of a $k$-dimensional face $F$ of $\Delta$. Choose a vertex $v$ of $F$. By reordering the normals $\nu_1,\ldots, \nu_d$ if necessary, we may assume that $v$ is characterized by the vanishing of $L_{1},\ldots,L_{m}$ and that $F$ is characterized by the vanishing of $L_{1},\ldots,L_{m-k}$. Since $\Delta$ is a Delzant polytope, the mapping 
$$
\mathfrak{t}^*\rightarrow \mathbb{R}^m:x\mapsto (L_1(x),\ldots,L_m(x))
$$
is an affine isomorphism. The coordinates $y=(y^i)$ defined by $y^i=L_i(x)-L_i(x_0)$ for $i=1\ldots,m$ are called {\bf adapted to $F$} (centered on $x_0$).
\end{déf}

\begin{prop}[\cite{apostolov:3}, Proposition 1]\label{CritèreAGCTF}
Let $(M,\omega, \mathbb{T})$ be a compact symplectic toric manifold of real dimension $2m$ with moment map $\mu:M\rightarrow\Delta\subset\mathfrak{t}^*$. A toric almost Kähler structure $\mathring{J}\in AK^{\mathbb{T}}_\omega(\mathring{M})$ defined on $\mathring{M}$ is the restriction of an element of $AK^{\mathbb{T}}_\omega(M)$ if and only if for each $k$-dimensional face $F$ of $\Delta$ with adapted coordinates $(y^i)$, the matrix $H_{ij}$, defined on $\mathring{\Delta}$ as the matrix whose inverse is $H^{ij}(\mu(p))=\mathring{g}_p(X_{\nu_i},X_{\nu_j})$ ($1\leq i,j\leq m$), satisfies the following conditions:
\begin{itemize}
 \item[(i)] $H_{ij}$ admits a smooth extension to $\Delta$;
 \item[(ii)] on each facet $F_i$ containing $F$,
$$
H_{ij}(y)=0 \quad \forall j=1,\ldots,m \quad \text{and} \quad \frac{\partial H_{ii}}{\partial y^i}=2;
$$
 \item[(iii)] the sub-matrix $(H_{ij})_{i,j=m-k+1}^{m}$ is positive definite on $\mathring{F}$ ($k>0$).
\end{itemize}
Alternatively, $\mathring{J}$ is the restriction of an element of $AK^{\mathbb{T}}_\omega(M)$ if and only if conditions (C1), (C2)' of Lemma \ref{ConditionsPrime} are satisfied.
\end{prop}

\begin{proof}[Proof of Lemma \ref{Lemmeg_f}]
If $H_{ij}$ is the matrix corresponding to $\mathring{g}_K$ as in the statement of Proposition \ref{CritèreAGCTF}, then $fH_{ij}$ is the matrix corresponding to $\mathring{g}_f$. It suffices to realize that conditions (i)-(iii) are satisfied by $H_{ij}$ if and only if they are satisfied by $fH_{ij}$. For (i) and (iii), it is trivial, while for (ii), we have 
\begin{flalign*}
\frac{\partial (fH_{ii})}{\partial y^i}(y) & =\frac{\partial f}{\partial y^i}(y)H_{ii}(y)+f(y)\frac{\partial H_{ii}}{\partial y^i}(y) \\
& = \frac{\partial H_{ii}}{\partial y^i}(y),
\end{flalign*}
using (i) and the hypothesis $f(y)=1$ for $y\in \partial\Delta$.
\end{proof}

Finally, we can prove the result announced at the begining of this section. 

\begin{thm}\label{ThmCompactificationDim4}
Consider $\mathring{J}\in DGK^{\mathbb{T}}_\omega(\mathring{M})$ of the form \eqref{FormeBlocAntidiagonale} with respect to admissible coordinates, where $\mathring{\Psi}=S+C$. Consider also the Riemannian metric $\mathring{g}=\omega(\cdot,\mathring{J}\cdot)^s$ and $\mathring{g}_K$ the toric almost Kähler metric on $\mathring{M}$ definned by equation \eqref{Défg_K}. Then, $\mathring{J}$ is the restriction of an element of $DGK^{\mathbb{T}}_\omega(M)$ if and only if $\mathring{g}_K$ is the restriction to $\mathring{M}$ of an $\omega$-compatible toric Kähler metric on $M$. In particular, this condition is equivalent to the following conditions for the matrix $H_{ij}$, defined on $\mathring{\Delta}$ as the matrix whose inverse is $H^{ij}(\mu(p))=\mathring{g}_K|_p(X_{\nu_i},X_{\nu_j})$ ($1\leq i,j\leq 2$):
For each $k$-dimensional face $F$ of $\Delta$ with adapted coordinates $(y^i)$,
\begin{itemize}
 \item[(i)] $H_{ij}$ admits a smooth extension to $\Delta$;
 \item[(ii)] on each facet $F_i$ containing $F$,
$$
H_{ij}(y)=H_{ji}(y)=0, \quad \forall j=1,2 \quad \text{and} \quad \frac{\partial H_{ii}}{\partial y^i}=2;
$$
 \item[(iii)] the sub-matrix $(H_{ij})_{i,j=m-k+1}^{m}$ is positive definite on $\mathring{F}$ ($k>0$).
\end{itemize}
Alternatively, $\mathring{J}$ is the restriction of an element of $DGK^{\mathbb{T}}_\omega(M)$ if and only if conditions (C1), (C2)' of Lemma \ref{ConditionsPrime} are satisfied.
\end{thm}

\begin{proof}
By Proposition \ref{CritèreAGCTF}, conditions (i)-(iii) are equivalent to the compactification of $\mathring{g}_K$. 

If $\mathring{J}$ is the restriction to $\mathring{M}$ of an element of $DGK^{\mathbb{T}}_\omega(M)$, the by Remark \ref{Remg_f} (1), $\mathring{g}_K$ is the restriction to $\mathring{M}$ of a toric Kähler metric on $M$.

Suppose next that conditions (i)-(iii) are met, and let us show that this implies conditions (C1), (C2)' of Lemma \ref{ConditionsPrime}. This will prove that conditions (i)-(iii) are sufficient to compactification and also that conditions (C1), (C2)' are necessary. By Proposition \ref{CritèreAGCTF}, (C1) is satisfied for $\mathring{g}_K$, \i.e. $S-\Psi$ admits a smooth extension to $\Delta$, where $\Psi$ comes from an element of $DGK_\omega^{\mathbb{T}}(M)$. Since the matrix-valued function $C$ is constant, it admits a smooth extension to $\Delta$ and so $\mathring{\Psi}-\Psi$ admits a smooth extension to $\Delta$, \i.e. (C1) is satisfied for $\mathring{J}$. For (C2)', we must show that for any point $x_0\in\partial\Delta$, we have
$$
\lim_{x\mapsto x_0}\frac{\det\mathring{\Psi}}{\det\Psi}\neq 0.
$$
Let $F$ be the face of $\Delta$ which contains $x_0$ in its interior. It is shown in the proof of Proposition \ref{CritèreAGCTF} that with respect to coordinates $y=(y^i)$ adapted to $F$ centered on $x_0$, we have
$$
(\det S(y))^{-1}=2^{m-k}y^1\ldots y^{m-k}\mathring{P}(y),
$$ 
$$
(\det\Psi(y))^{-1}=2^{m-k}y^1\ldots y^{m-k}P(y),
$$
where $k$ is the dimension of $F$ and where $P, \mathring{P}\in C^{\infty}(\Delta)$ are smooth function, positive at $y=0$. It follows that 
$$
\lim_{x\mapsto x_0}\frac{\det\mathring{\Psi}}{\det\Psi}= \lim_{y\mapsto 0}\left(\frac{\det S(y)}{\det\Psi(y)}+\frac{c^2}{\det\Psi(y)}\right)=\frac{P(0)}{\mathring{P}(0)}>0.
$$
\end{proof}

As a corollary, we have the converse of Corollary \ref{JExtrémaleDim4}.

\begin{cor}\label{JExtrémaleDim4ForReal}
Let $(M,\omega,\mathbb{T},\mu)$ be a compact symplectic toric manifold of dimension 4. If there exists an extremal element $J\in DGK_{\omega}^{\mathbb{T}}(M)$, then there exists extremal elements $J_0\in K_{\omega}^{\mathbb{T}}(M)$.
\end{cor}

\begin{proof}
Let $J\in DGK_{\omega}^{\mathbb{T}}(M)$ be extremal and consider the Kähler structure $g_K$ associated with it in the sense of Theorem \ref{ThmCompactificationDim4}. Clearly, the generalized Hermitian scalar curvature $u_{GK}(J)$ of $J$ is equal to the scalar curvature $s_{g_K}$ of $g_K$ so that, by Proposition \ref{CaracterisationsExtremal}, $u_{GK}(J)$ is an affine function in $(\mu^1,\mu^2)$. By \cite{abreu:1}, this property characterizes toric extremal Kähler metric.
\end{proof}

\begin{rem}
{\rm
In dimension 4, we can use Theorem \ref{ThmCompactificationDim4} to enlarge the scope of \eqref{ExprAlt} to the case of the action of the full group $\mathrm{Ham}^{\mathbb{T}}(M,\omega)$. Specifically, if $(M,\omega, \mathbb{T})$ is a compact symplectic toric manifold of real dimension $4$ and $\nu:DGK_{\omega}^{\mathbb{T}}(M)\rightarrow (C_0^{\infty}(M)^{\mathbb{T}})^*$ is the function given by \eqref{ExprAlt}, then for any $J\in DGK_{\omega}^{\mathbb{T}}(M)$ et $f\in C_0^{\infty}(M)^{\mathbb{T}}$, we have
$$
d(\nu^f)_J(\dot{J})=-\Omega_J(V^{\sharp}_{\; J},\dot{J}),
$$
where $V=\mathrm{grad}_{\omega}f$ and $V^{\sharp}$ is the corresponding vector field on $AGK^{\mathbb{T}}_\omega(M)$ given by \eqref{Vfondamental}.
}
\end{rem}

\subsection{An explicit formula for the generalized Hermitian scalar curvature in dimension 4}\label{sectionfinale}

Since in the Kähler situation $J=-J^{*_\omega}$, the right hand side in equation \eqref{ScalGKPot} corresponds to the scalar curvature of the associated Riemannian metric, it is natural to try to relate $u_{GK}(J)$ to the scalar curvatures of the corresponding Hermitian structure $(J,g)$, $(J^{*_\omega},g)$. Henceforth, let $(M,\omega, \mathbb{T})$ a compact symplectic toric manifold of real dimension $4$ with moment map $\mu:M\rightarrow \Delta\subset \mathfrak{t}^*$ and consider $J\in DGK^{\mathbb{T}}_\omega(M)$ of the form \eqref{FormeBlocAntidiagonale} with respect to angular coordinates $t^j$ on $\mathring{M}$.

Recall that given an almost Hermitian structure $(g,J)$ on $M$ with Chern connection $\nabla$, the induced Hermitian connection $\hat{\nabla}$ on the anticanonical bundle $\bigwedge^2(T^{1,0}M)$ has curvature $R^{\hat{\nabla}}=\sqrt{-1}\rho^{\nabla}\otimes\mathrm{Id}$ where the real 2-form $\rho^{\nabla}$ is called the {\em Chern-Ricci form} of the almost Hermitian structure. The {\em Hermitian scalar curvature} of the almost Hermitian structure is
\begin{equation}\label{uAlternative}
u=\frac{4\rho^\nabla\wedge F}{F\wedge F},
\end{equation}
where $F=gJ$. If $J$ is integrable, $\hat{\nabla}$ is the Chern connection on the anticanonical bundle relative to the induced Hermitian metric and its natural holomorphic structure, and the Ricci-Chern form admits the following local expression:
\begin{equation}\label{RicciHermLoc}
\rho^\nabla=-\frac{1}{2}dd^c\log\sqrt{\det(g_{ij})},
\end{equation}
where $g_{ij}$ are the components of $g$ relative to local holomorphic coordinates. Since the metric $g$ coming from a generalized Kähler structure of symplectic type $(J_+,J_-,g,b)$ is Gauduchon (cf. section \ref{section2}), the Hermitian scalar curvatures $u_{\pm}$ of the Hermitian pairs $(g,J_{\pm})$ are related to the Riemannian scalar curvature $s$ of $g$ by the Lee forms \cite{gauduchon:3}:
\begin{equation}\label{Scal-u}
u_{\pm}=s+\frac{1}{2}|\theta_{\pm}|^2.
\end{equation}
In particular, since $|\theta_+| = |\theta_-|$, $u_+= u_-$, a common function which we denote by $u_J$ and refer to as the {\em Hermitian scalar curvature of the generalized Kähler structure}.

The following technical lemmas will be used in the proof of Theorem \ref{ThmPrincipal} below. In proving them, we shall make use of formulas \eqref{partiesSymAsym}-\eqref{RelationDet} as well as the relation
\begin{equation}\label{FormesVolumes}
v_g=\frac{\det S}{\det \Psi}v_\omega
\end{equation}
between the volume forms induced by $g$ and $\omega$ respectively. 

\begin{lemme}\label{IdentitéMagique} The matrix $S$ satisfies the identity
$$
\sum\limits_{i=1}^2(\det S)S^{ij}_{\;\;,i}=-\sum\limits_{i=1}^2(\det S)_{,i}S^{ij}, \quad j=1,2.
$$
\end{lemme}

\begin{proof}
It suffices to differentiate the identity $(\det S)S^{-1}=CSC^{-1}$:
\begin{flalign*}
\sum\limits_{i=1}^2((\det S)S^{ij})_{,i} & =\sum\limits_{i=1}^2(CSC^{-1})_{ij,i} =\sum\limits_{i,\alpha,\beta=1}^2C_{i\alpha}S_{\alpha\beta,i}C^{\beta j} =\sum\limits_{i,\alpha\beta=1}^2C_{i\alpha}S_{\alpha i,\beta}C^{\beta j}.
\end{flalign*}
But $\sum\limits_{i,\alpha=1}^2C_{i\alpha}S_{\alpha i,\beta} = \mathrm{tr}(CS_{,\beta})$ which vanishes since $C$ is antisymmetric and $S_{,\beta}$ is symmetric.
\end{proof}

\begin{lemme}\label{2eLemme}
For the angle function $p=-\frac{1}{4}\mathrm{tr}(JJ^{*_\omega})$ and the Lee form $\theta$ of the Hermitian pair $(\omega,J)$, we have:
$$
\Delta p=\sum\limits_{i,j=1}^2\frac{2c^2}{(\det \Psi)^2}\left((\det S)_{,ij}-\frac{3}{\det \Psi}(\det S)_{,i}(\det S)_{,j}\right)S^{ij},
$$
$$
|\theta|^2=\sum\limits_{i,j=1}^2\frac{c^2(\det S)_{,i}(\det S)_{,j}}{(\det \Psi)^2\det S}S^{ij}.
$$
\end{lemme}

\begin{proof}
Using identity \eqref{RelationDet}, we compute
$$
dp=\frac{-2c^2}{(\det\Psi)^2}\sum\limits_{i=1}^2(\det S)_{,i}d\mu^i.
$$
We have
\begin{flalign*}
*d\mu^i = & S^{1i}\frac{\det S}{\det\Psi}dt^1\wedge d\mu^2\wedge dt^2 +S^{2i}\frac{\det S}{\det\Psi}d\mu^1\wedge dt^1\wedge dt^2 \\
&+ S_{1i}d\mu^1\wedge d\mu^2\wedge dt^2 + S_{2i}d\mu^1\wedge dt^1\wedge d\mu^2.
\end{flalign*}
Thus, using the formula from Lemma \ref{IdentitéMagique},
\begin{flalign*}
\Delta p &=-*d\left(\frac{\det S}{\det\Psi}\sum\limits_{i,j=1}^2p_{,i}(S^{1i}dt^1\wedge d\mu^2\wedge dt^2 + S^{2i}d\mu^1\wedge dt^1\wedge dt^2)\right) \\
&=*\frac{2c^2}{(\det \Psi)^2} \sum\limits_{i,j=1}^2\left((\det S)_{,ij} - \frac{3}{\det \Psi}(\det S)_{,i}(\det S)_{,j}S^{ij}\right)\frac{\det S}{\det \Psi}v_\omega.
\end{flalign*}
The desired formula then follows from \eqref{FormesVolumes}. To compute $|\theta|^2$, we use the result from \cite{apostolov:4} (Lemma 7) according to which $dp=\frac{1}{2}[J,J^{*_\omega}]^*\theta$.
Leaning on \eqref{IdentitéCruciale}, we easily show that $[J,J^{*_\omega}]^2=4(p^2-1)\mathrm{Id}$, which allows us to solve for $\theta$ in the preceding formula:
\begin{equation}\label{Theta-dp}
\theta=\frac{1}{2(p^2-1)}[J,J^{*_\omega}]^*dp.
\end{equation}
Using the fact that $[J,J^{*_\omega}]$ is $g$-antisymmetric, we compute
\begin{flalign*}
|\theta|^2&=\frac{(\det \Psi)^2}{4c^2\det S}|dp|^2 =\sum\limits_{i,j=1}^2\frac{c^2}{(\det \Psi)^2\det S}(\det S)_{,i}(\det S)_{,j}S^{ij}.
\end{flalign*}
\end{proof}

\begin{thm}\label{ThmPrincipal}
Let $(M,\omega, \mathbb{T})$ be a compact symplectic toric manifold of real dimension $4$ with moment map $\mu:M\rightarrow \Delta\subset \mathfrak{t}^*$. Consider $J\in DGK_{\omega}^{\mathbb{T}}(M)$ of the form \eqref{FormeBlocAntidiagonale} with respect to angular coordinates $t^j$ on $\mathring{M}$ (cf. Proposition \ref{LemmeCoordonnéesCX}) and let $g$ be the symmetric part of $\omega(\cdot,J\cdot)$. The Hermitian Ricci form of the Hermitian structure $(g, J)$ is given on $\mathring{M}$ by
\begin{equation*}
\rho^{\nabla}=-\frac{1}{2}dd^c\log \det S + dd^c\log \det \Psi,
\end{equation*}
where $\Psi=S+C$ is the decomposition of $\Psi$ into its symmetric and antisymmetric part. The Hermitian scalar curvature of the generalized Kähler structure $J$ is
\begin{equation}\label{FormulePrincipale}
u_J = -\sum\limits_{i,j=1}^2S^{ij}_{\;\;,ij}+\frac{4-2p}{1-p}|\theta|^2-\frac{2\langle[J,J^{*_\omega}],d\theta\rangle}{1-p},
\end{equation}
where $p$ is the angle function (cf. Lemma \ref{LemmePontecorvo}) and where $[J,J^{*_\omega}]$ is seen as a 2-form by means of the metric $g$.
\end{thm}

\begin{proof}
We shall compute $u_J$ from \eqref{uAlternative} and $\rho^\nabla$ from \eqref{RicciHermLoc}. In terms of the local holomorphic coordinates $(u^j,t^j)$ from Proposition \ref{LemmeCoordonnéesCX}, we have 
$$
g=\sum\limits_{i,j=1}^2((\Psi^{-1})^TS\Psi^{-1})_{ij}du^i\otimes du^j + \frac{\det S}{\det\Psi}\sum\limits_{i,j=1}^2S^{ij}dt^i\otimes dt^j,
$$
whence $\sqrt{\det(g_{ij})}=\frac{\det S}{(\det\Psi)^2}$ and so
\begin{equation}\label{rho}
\rho^{\nabla}=-\frac{1}{2}dd^c\log \det S + dd^c\log \det \Psi,
\end{equation}
$$
u_J=\frac{4\rho^{\nabla}\wedge F}{F\wedge F}=\frac{4dd^c\log \det \Psi\wedge F}{F\wedge F}-\frac{2dd^c\log \det S\wedge F}{F\wedge F}.
$$
To devellop the first term, we use the general formula $(\det A)' = (\det A)\mathrm{tr}\left(A^{-1}A'\right)$
for the $t$-derivative of the determinant of a non-singular matrix $A=A(t)$. In particular,
$$
(\log\det \Psi)_{,i}=\sum\limits_{\alpha,\beta=1}^2\Psi^{\alpha\beta}\Psi_{\beta\alpha,i},
$$
which yields
\begin{flalign*}
dd^c\log \det \Psi & = \sum\limits_{i,j,k=1}^2\left(( \log \det \Psi)_{,i}\Psi^{ij}\right)_{,k}d\mu^k\wedge dt^j \\
& = \sum\limits_{i,j,k,\alpha,\beta=1}^2\left(\Psi^{\beta\alpha}\Psi_{\alpha\beta,i}\Psi^{ij}\right)_{,k}d\mu^k\wedge dt^j \\
& = -\sum\limits_{i,j,k,\alpha,\beta=1}^2\left(\Psi^{\beta\alpha}\Psi_{\alpha i}\Psi^{ij}_{\;\;,\beta}\right)_{,k}d\mu^k\wedge dt^j\\
& = -\sum\limits_{i,j,k=1}^2\Psi^{ij}_{\;\;,ik}d\mu^k\wedge dt^j\\
\end{flalign*}
For the second equality, we have used the fact that $\Psi_{\alpha\beta,i}=S_{\alpha\beta, i}= S_{\alpha i,\beta }= \Psi_{\alpha i,\beta}$ (since $S$ is a Hessian), and also the 
identity
$$
\sum\limits_{i=1}^2 \Psi_{\alpha i,\beta}\Psi^{ij}=-\sum\limits_{i=1}^2 \Psi_{\alpha i}\Psi^{ij}_{\;\;,\beta}
$$
obtained by differentiating $\Psi\Psi^{-1} = I$ with respect to par $\mu^{\beta}$. Finally, note that $F$ is the $(1,1)$ part of $\omega$ with respect to $J$. And since $\bigwedge^{3,1}=\bigwedge^{1,3}=0$ in dimension 4, we have
$$
dd^c\log\det \Psi\wedge F=dd^c\log\det \Psi\wedge \omega = -\frac{1}{2}\sum\limits_{i,j=1}^2\left(\frac{\det S}{\det \Psi}S^{ij}\right)_{,ij}\omega\wedge\omega.
$$
Using $F\wedge F=\frac{\det S}{\det \Psi}\omega\wedge\omega$, we obtain
$$
\frac{4dd^c\log \det \Psi\wedge F}{F\wedge F}=-2\frac{\det \Psi}{\det S}\sum\limits_{i,j=1}^2\left(\frac{\det S}{\det \Psi}S^{ij}\right)_{,ij}.
$$

For the second term, we proceed as follows.
$$
dd^c\log\det S = \sum\limits_{i,j,k=1}^2\left(\frac{(\det S)_{,i}}{\det S}\Psi^{ij}\right)_{,k}d\mu^k\wedge dt^j,
$$
so
\begin{flalign*}
\frac{-2dd^c\log\det S \wedge F}{F\wedge F} & =\frac{\det \Psi}{\det S}\left(\frac{-2dd^c\log\det S \wedge \omega}{\omega\wedge\omega}\right) \\
& = -\frac{\det \Psi}{\det S}\sum\limits_{i,j=1}^2\left(\frac{(\det S)_{,i}}{\det S}\Psi^{ij}\right)_{,j}.
\end{flalign*}
But, here also, we have $\Psi^{ij}=\frac{\det S}{\det \Psi}S^{ij}-\frac{1}{\det \Psi}C_{ij}$ with
$$
\sum\limits_{i,j=1}^2\left(\frac{(\det S)_{,i}}{\det S}C^{ij}\right)_{,j}=\sum\limits_{i,j=1}^2\frac{(\det S)_{,ij}}{\det S}C_{ij}-\sum\limits_{i,j=1}^2\frac{(\det S)_{,i}(\det S)_{,j}}{(\det S)^2}C_{ij}=0,
$$
so using the identity from Lemma \ref{IdentitéMagique}, 
\begin{flalign*}
\frac{-2dd^c\log\det S \wedge F}{F\wedge F} & =-\frac{\det \Psi}{\det S}\sum\limits_{i,j=1}^2\left(\frac{(\det S)_{,i}}{\det \Psi}S^{ij}\right)_{,j} \\
& =\frac{\det \Psi}{\det S}\sum\limits_{i,j=1}^2\left(\frac{\det S}{\det \Psi}S^{ij}_{\;\;,i}\right)_{,j}.
\end{flalign*}
We thus obtain
\begin{flalign}
\notag u_J & =\frac{\det \Psi}{\det S}\sum\limits_{i,j=1}^2\left(-2\left(\frac{\det S}{\det \Psi}S^{ij}\right)_{,ij}+\left(\frac{\det S}{\det \Psi}S^{ij}_{\;\;,i}\right)_{,j}\right) \\
\notag & =\frac{\det \Psi}{\det S}\sum\limits_{i,j=1}^2\left(-\left(\frac{\det S}{\det \Psi}\right)S^{ij}_{\;\;,ij} -3\left(\frac{\det S}{\det \Psi}\right)_{,i}S^{ij}_{\;\;,j}-2\left(\frac{\det S}{\det \Psi}\right)_{,ij}S^{ij}\right) \\
& \label{s_GK-Deltap}=\sum\limits_{i,j=1}^2-S^{ij}_{\;\;,ij}-\frac{2\Delta p}{1-p}+\frac{4+2p}{1-p}|\theta|^2,
\end{flalign}
by using Lemma \ref{IdentitéMagique} on the middle term of the second expression. Finally, we invoke formula (26) from \cite{apostolov:4} which reads
$$
\Delta p = 2p|\theta|^2 + \langle[J,J^{*_\omega}],d\theta\rangle
$$
to land on the announced formula.
\end{proof}

\begin{cor}\label{CoroS_GK}
The generalized Hermitian scalar curvature of $J$ admits the following expression (which depends only on $J$ and $\omega$)
$$
u_{GK}(J)=u_J -\frac{4-2p}{1-p}|\theta|^2+\frac{2\langle[J,J^{*_\omega}],d\theta\rangle}{1-p}.
$$
Alternatively, we may write
\begin{equation}\label{s_GK-Scal}
u_{GK}(J)=s_g+\frac{2\Delta p}{1-p}-\frac{1}{1-p^2}\left(\frac{4+2p}{1-p}-\frac{1}{2}\right)|dp|^2,
\end{equation}
where $s_g$ is the scalar curvature of the associated Riemannian metric $g=\omega(\cdot,J\cdot)^s$.
\end{cor}

\begin{proof}
The first expression is obtained trivially by comparing formula \eqref{FormulePrincipale} with the definition of $u_{GK}(J)$. The second expression is obtained similarly from \eqref{s_GK-Deltap} by using \eqref{Scal-u} as well as the identity
\begin{equation}\label{NormesTheta-dp}
|\theta|^2=\frac{1}{1-p^2}|dp|^2
\end{equation}
which one deduces from \eqref{Theta-dp}.
\end{proof}

\begin{rem}
{\rm
In \cite{coimbra:1} a notion of generalized scalar curvature depending on an arbitrary function $\phi\in C^{\infty}(M)$ and valid in all dimensions is invented. This expression takes the following form (\cite{garcia-fernandez:1} p.22):
$$
GS^{\phi}(J)=s_g+4\Delta\phi-4|d\phi|^2-\frac{1}{2}|db|^2.
$$
In dimension 4, we have shown in Lemma \ref{GK-Gauduchon} that $db=*\theta$ and so $|db|^2=|\theta|^2=(1-p^2)^{-1}|dp|^2$ (by \eqref{NormesTheta-dp}). 
Comparing with \eqref{s_GK-Scal}, we conclude that $GS^{\phi}(J)=u_{GK}(J)$ if and only if $\phi=-\frac{1}{2}\log(1-p)$, which suggest a prefered choice for the function $\phi$ of \cite{garcia-fernandez:1}.
}
\end{rem}

{}
\begin{thebibliography}{}

\bibitem{abreu:1}
 {\scshape M. Abreu},
 \emph{Kähler Geometry of Toric Varieties and Exremal Metrics}, Internat. J. Math., {\bf 9} (1998), 641-651.

 \bibitem{abreu:2}
 {\scshape M. Abreu},
 \emph{Kähler Geometry of toric manifolds in symplectic coordinates}, in 'Symplectic and Contact Topology: Interactions and Perspectives' (eds. Y. Eliashberg, B. Khesin and F. Lalonde), Fields Institute Communications {\bf 35}, American Mathematical Society, 2003.

 \bibitem{apostolov:2}
 {\scshape V. Apostolov, M. Gualtieri},
 \emph{Generalized Kaehler manifolds, commuting complex structures, and split tangent bundles}, Comm. Math. Phys. {\bf 271} (2007), 561-575.

 \bibitem{apostolov:3}
 {\scshape V. Apostolov, D. Calderbank, P. Gauduchon, C. Tønnesen-Friedman},
 \emph{Hamiltonian 2-forms in Kaehler Geometry II: Global Classification}, J. Differential Geom. {\bf 68} (2004), 277-345.

 \bibitem{apostolov:4}
 {\scshape V. Apostolov, P. Gauduchon, G. Grantcharov},
 \emph{Bihermitian structures on complex surfaces}, Proc. London Math. Soc. {\bf 79} (1999), 414-429 + Erratum in Proc. London Math. Soc. {\bf 92} (2006), 200-202.

\bibitem{atiyah:1}
 {\scshape M. Atiyah},
 \emph{Convexity and commuting Hamiltonians}, Bull. London Math. Soc. {\bf 14} (1982), 1-15.

\bibitem{banyaga:1}
 {\scshape A. Banyaga},
 \emph{Sur la structure du groupe des difféomorphismes qui préservent une forme symplectique}, Comm. Math. Helv. {\bf 53} (1978), 174-227.

\bibitem{calabi:1}
 {\scshape E. Calabi},
 \emph{Extremal Kähler metrics}, Seminar of Differerential Geometry, ed. S.T. Yau, Annals of Mathematics Studies {\bf 102}, Princeton University Press (1982), 259-290.

\bibitem{chen:1}
 {\scshape B. Chen, A.-M. Li, L. Sheng},
 \emph{Extremal metrics on toric surfaces}, arXiv:1008.2607

\bibitem{coimbra:1}
 {\scshape C. Coimbra, C. Strickland-Constable, D. Waldran},
 \emph{Supergravity as Generalized Geometry I: Type II Theories}, Jour. High Energy Phys. {\bf 91} (2011)

\bibitem{delzant:1}
 {\scshape T. Delzant},
 \emph{Hamiltoniens périodiques et image convexe de l'application moment}, Bull. Soc. Math. France {\bf 116} (1998), 315-339.

\bibitem{donaldson:1}
 {\scshape S. K. Donaldson},
 \emph{Remarks on gauge theory, complex geometry and 4-manifold topology}, Fields Medallists' lectures, 384-403, World Sci. Ser. 20th Century Math., 5, World Sci. Publishing, River Edge, NJ, 1997.

 \bibitem{donaldson:2}
  {\scshape S. K. Donaldson},
  \emph{Scalar curvature and stability of toric varieties}, J. Differential Geom. {\bf 62} (2002), 289-349.

\bibitem{donaldson:3}
 {\scshape S. K. Donaldson},
 \emph{Constant scalar curvature metrics on toric surfaces}, GAFA, Geom. func. anal. {\bf 19} (2009), 83-136.

\bibitem{enrietti:1}
 {\scshape N. Enrietti, A. Fino, G. Grantcharov},
 \emph{Tamed symplectic forms and generalized geometry}, J. Geom. Phys. {\bf 71} (2013), 103-116.

\bibitem{fujiki:1}
 {\scshape A. Fujiki},
 \emph{Moduli space of polarized algebraic manifolds and Kähler metrics}, [translation of Sugaku 42, no. 3 (1990), 231-243], Sugaku Expositions 5, no. 2 (1992), 173-191.

\bibitem{garcia-fernandez:1}
  {\scshape M. Garcia-Fernandez},
  \emph{Torsion-free generalized connections and heterotic supergravity}, Comm. Math. Phys. {\bf 332} (2014) iss. 1, 89-115.

\bibitem{gauduchon:1}
 {\scshape P. Gauduchon},
 \emph{Structures bihermitiennes en dimension 4}, unpublished notes.

 \bibitem{gauduchon:2}
 {\scshape P. Gauduchon},
 \emph{Calabi's extremal Kähler metrics: An elementary introduction.} In preparation.

  \bibitem{gauduchon:3}
 {\scshape P. Gauduchon},
 \emph{La 1-forme de torsion d'une variété hermitienne compacte}, Math. Ann. {\bf 267} (1984), 495-518.

\bibitem{goto:1}
 {\scshape R. Goto},
 \emph{Deformations of generalized complex and generalized Kähler structures}, J. Differential Geom. {\bf 84} (2010), 525-560.
 
 \bibitem{goto:2}
 {\scshape R. Goto},
 \emph{Poisson structures and generalized Kahler structures}, J. Math. Soc. Japan, {\bf 61} (2009) no. 1, 107-132.

\bibitem{gualtieri:1}
  {\scshape M. Gualtieri},
  \emph{Generalized complex geometry}, DPhil thesis, Oxford University, 2004.

 \bibitem{gualtieri:2}
  {\scshape M. Gualtieri},
  \emph{Branes on Poisson Varieties}, The many facets of geometry, 368-394, Oxford Univ. Press, Oxford, 2010.
  
 \bibitem{gualtieri:4}
  {\scshape H. Bursztyn, G. R. Cavalcanti, and M. Gualtieri},
  \emph{Reduction of Courant algebroids and generalized complex structures}, Adv. Math. {\bf 211} (2007), no. 2, 726-765.

\bibitem{guillemin:1}
 {\scshape V. Guillemin},
 \emph{Kähler structures on toric varieties}, J. Differential Geom. {\bf 40} (1994), 285-309.
 
 \bibitem{guillemin:2}
  {\scshape V. Guillemin},
  \emph{Moment maps and Combinatorial Invariants of Hamiltonian $T^n$-spaces}, Progress in Mathematics {\bf 122}, Birkäuser, Boston, 1994.
 
 \bibitem{guillemin:3}
 {\scshape V. Guillemin, S. Sternberg},
 \emph{Convexity properties of the moment mapping I}, Invent. Math. {\bf 67} (1982), 491-513.

\bibitem{ginzburg:1}
 {\scshape L. Ginzburg, V. Guillemin, Y. Karshon},
 \emph{Moment maps, Cobordisms, and Hamiltonian Group Actions}, Mathematical Surveys and Monographs {\bf 98} (2002).

\bibitem{hitchin:1}
 {\scshape N. J. Hitchin},
 \emph{Instantons, Poisson structures and generalized Kähler geometry}, Comm. Math. Phys. {\bf 265} (2006), 131-164.

 \bibitem{hitchin:2}
 {\scshape N. J. Hitchin},
 \emph{Bihermitian metrics on Del Pezzo surfaces}, J. Symplectic Geom. {\bf 5} (2007), 1-7.

 \bibitem{hitchin:3}
  {\scshape N. J. Hitchin},
  \emph{Generalized Calabi-Yau manifolds}, Quart. J. Math. Oxford Ser. {\bf 54} (2003), 281-308.

   \bibitem{hitchin:4}
  {\scshape N. J. Hitchin, A. Karlhede, U. Lindström, M. Ro\v{c}ek},
  \emph{Hyperkahler Metrics and Supersymmetry}, Comm. Math. Phys. {\bf 108} (1987), no. 4, 535-589.
 
   \bibitem{hitchin:5}
  {\scshape N. J. Hitchin},
  \emph{Lectures on generalized geometry}, arXiv:1008.0973

\bibitem{hu:1}
 {\scshape S. Hu},
 \emph{Hamiltonian symmetries and reduction in generalized geometry}, Houston J. Math, 35 (3) (2009), 787-811.

\bibitem{lejmi:1}
 {\scshape M. Lejmi},
 \emph{Extremal almost-Kähler metrics}, Int. J. Math. {\bf 21} (2010), no. 12, 1639-1662. 

\bibitem{lerman:1}
 {\scshape E. Lerman, S. Tolman},
 \emph{Hamiltonian torus action on symplectic orbifolds and toric varieties}, Trans. Amer. Math., {\bf 184} (1997), 4201-4230.

\bibitem{lin:1}
 {\scshape Y. Lin, S. Tolman},
 \emph{Symmetries in Generalized Kähler Geometry}, Commun. Math. Phys. {\bf 268}, (2006) 199-122.

\bibitem{pontecorvo:1}
 {\scshape M. Pontecorvo},
 \emph{Complex structures on Riemannian four-manifolds}, Math. Ann. {\bf 309} (1997), no. 1, 159-177.

\bibitem{zhou:1}
 {\scshape B. Zhou, X. Zhu},
 \emph{A note on the K-stability on toric manifolds}, Proc. Amer. Math. Soc. {\bf 136} (2008), 3301-3307.

 

\end{thebibliography}
\end{document}